\documentclass{amsart}
\usepackage{latexsym, amscd, amsfonts, mathrsfs, amsmath, amssymb, amsthm, stmaryrd, tikz-cd, mathrsfs, bbm, url}

\def\p{\prime}
\def\pp{{\prime\prime}}
\def\bZ{\mathbb Z}

\def\bQ{\mathbb Q}

\def\inj{\hookrightarrow}

\def\es{\emptyset}

\def\bop{\bigoplus}

\def\sur{\twoheadrightarrow}

\def\sm{\sigma}

\def\imp{\Rightarrow}

\def\t{\times}

\def\iff{\Leftrightarrow}
\def\xrw{\xrightarrow}

\def\de{\delta}
\def\De{\Delta}

\def\Ga{\Gamma}
\def\ga{\gamma}

\DeclareMathOperator\sgn{\mathrm{sgn}}

\def\ep{\epsilon}

\def\ze{\zeta}

\DeclareMathOperator\rk{\mathrm{rank}}

\def\der{\partial}

\DeclareMathOperator\MC{\mathrm{MC}}
\DeclareMathOperator\MH{\mathrm{MH}}
\DeclareMathOperator\Shr{\mathrm{Shr}}
\DeclareMathOperator\Dod{\mathrm{Dod}}
\DeclareMathOperator\Des{\mathrm{Des}}

\newcommand{\rom}[1]{\textup{\uppercase\expandafter{\romannumeral#1}}}

\newtheorem{thm}{Theorem}[section]
\newtheorem{lemma}[thm]{Lemma}
\newtheorem{prop}[thm]{Proposition}
\newtheorem{coro}[thm]{Corollary}

\theoremstyle{definition}

\newtheorem{defn}[thm]{Definition}
\newtheorem{eg}[thm]{Example}

\newtheorem{notn}[thm]{Notation}

\newtheorem{rmk}[thm]{Remark}

\begin{document}
  \title{Graph magnitude homology via algebraic Morse theory}
  \author{Yuzhou Gu}
  \begin{abstract}
  We compute magnitude homology of various graphs using algebraic Morse theory. 
  Specifically, we (1) give an alternative proof that trees are diagonal, (2) identify a new class of diagonal graphs, (3) prove that the icosahedral graph is diagonal, and (4) compute the magnitude homology of cycles. These results answer several questions of Hepworth and Willerton \cite{HW17}.
  \end{abstract}
  \maketitle
  \tableofcontents

  \section{Introduction}
  \subsection{Background}
  The magnitude of a finite metric space is a cardinality-like invariant defined and first studied by Leinster \cite{Lei13}. It is a special case of a general theory of magnitude of an enriched category, and has found applications in areas like biodiversity (e.g., Leinster and Cobbold \cite{LC12}).
  
  A finite graph\footnote{In this paper, we consider only finite simple undirected graphs. We often also assume that a graph is connected.} naturally gives rise to a finite metric space, and therefore has magnitude associated with it.
  The magnitude of a graph can be represented as a power series with integer coefficients.
  Leinster \cite{Lei17} studied magnitude of graphs and proved many interesting properties, such as multiplicativity with respect to Cartesian products, inclusion-exclusion formula under certain conditions, and invariance under Whitney twists with adjacent gluing points.
  
  Magnitude admits a categorification, called magnitude homology, in the sense that a coefficient of the power series is the Euler characteristic of corresponding homology groups.
  Magnitude homology is defined for graphs by Hepworth and Willerton \cite{HW17} and for enriched categories by Leinster and Shulman \cite{LS17}.
  Hepworth and Willerton proved that magnitude homology admits properties that categorify properties of magnitude. A K\"unneth theorem categorifies multiplicativity, and a Mayer-Vietoris theorem categorifies inclusion-exclusion formula.\footnote{We should note that it is still open whether magnitude homology is invariant under Whitney twists with adjacent gluing points.}
  They also proved a theorem which essentially computes the magnitude homology of joins of graphs.
  
  Using these theorems, Hepworth and Willerton are able to compute magnitude homology of many graphs, including trees, complete multipartite graphs, and so on.
  On the other hand, it turns out magnitude homology of graphs can be difficult to compute, even for very simple graphs. Based on computer computations of the first homology groups, Hepworth and Willerton made explicit conjectures for cycle graphs and the icosahedral graph.
  
  Algebraic Morse theory, developed independently by J\"ollenbeck \cite{Jol05} and by Sk\"oldberg \cite{Sko06}, is a useful combinatorial tool for homology computations. Many successful computations have been done using algebraic Morse theory, such as cohomology of certain nilpotent Lie algebras by Sk\"oldberg, and Hochschild homology of certain algebras by J\"ollenbeck.
  
  \subsection{Our results}
  In this paper, we use algebraic Morse theory to compute magnitude homology of various graphs.
  \subsubsection{Trees}
  Hepworth and Willerton proved that trees are diagonal (\cite{HW17}, Corollary 31). Their proof relies on a Mayer-Vietoris theorem, which can take some effort to prove. As a warmup for more complicated computations, we give another proof of this fact using algebraic Morse theory (Proposition \ref{PropTree}).
  \subsubsection{A new class of diagonal graphs}
  Hepworth and Willerton proved that joins of graphs are diagonal (\cite{HW17}, Theorem 37).
  We define a new class of graphs, named pawful graphs (Definition \ref{DefnPawful}), which strictly contains the class of joins, and prove that pawful graphs are diagonal (Theorem \ref{ThmPawful}).
  \subsubsection{Icosahedral graph}
  Based on computer computations, Hepworth and Willerton \cite{HW17} conjectured that the icosahedral graph is diagonal. We prove this using algebraic Morse theory (Theorem \ref{ThmIco}).

  \subsubsection{Cycles}
  Based on computer computations, Hepworth and Willerton \cite{HW17} made conjectures about ranks of magnitude homology groups of cycles.
  We prove their conjectures by computing using algebraic Morse theory (Theorem \ref{ThmOddCycle} for odd cycles and Theorem \ref{ThmEvenCycle} for even cycles).
  \subsubsection{Magnitude homology is stronger than magnitude}
  Hepworth and Willerton \cite{HW17} asked whether there exist graphs with the same magnitude but different magnitude homology. In Appendix \ref{SecMHvsMag}, we answer the question in the affirmative by giving explicit examples.
  
  \subsubsection{Geodetic ptolemaic graphs}
  Slightly generalizing the proof of Proposition \ref{PropTree}, we prove that graphs that are both geodetic and ptolemaic are diagonal. However, this does not give new diagonal graphs, because a graph is both geodetic and ptolemaic if and only if it is a block graph, whose diagonality follows from Mayer-Vietoris.
  We study geodetic ptolemaic graphs in Appendix \ref{SecGeoPto}.
  
  
  \subsection{Organization of the paper}
  The paper is organized as follows. In Section \ref{SecPre} we recall basic definitions and results about magnitude homology and algebraic Morse theory.
  In Section \ref{SecDes} we study special kinds of matchings, and prove several results that can simplify proofs of correctness of matchings.
  In Section \ref{SecComp} we carry out the computations and prove the main results.
  In Appendix \ref{SecMHvsMag} we give examples of graphs with the same magnitude but different magnitude homology.
  In Appendix \ref{SecGeoPto} we prove using algebraic Morse theory that graphs that are both geodetic and ptolemaic are diagonal, and explain why this does not give a new class of diagonal graphs.
  
  \subsection{Acknowledgements}
  The author is partially supported by Jacobs Family Presidential Fellowship during the preparation of this paper.
  The author would like to thank Richard Hepworth, Yury Polyanskiy, and Simon Willerton for helpful discussions.
  
  \section{Preliminaries}\label{SecPre}
  In this section we review necessary definitions and results regarding magnitude homology and algebraic Morse theory.
  We do not state them in full generality, but in a generality that suffices for our purposes.
  \subsection{Magnitude homology}
  This part follows Hepworth and Willerton \cite{HW17}.
  Let $G$ be a finite simple undirected connected graph.
  For a sequence of vertices $(x_0,\ldots,x_k)$, let $$\ell(x_0,\ldots,x_k) = \sum_{0\le i\le k-1} d(x_i,x_{i+1}).$$
  \begin{defn}[Magnitude homology]
  The magnitude chain complex $\MC_{*,*}(G)$ is defined as
  \begin{align*}
  \MC_{k,l}(G) = \bZ \{(x_0,\ldots,x_k)\in V(G)^{k+1}: x_i\ne x_{i+1}\forall i, \ell(x_0,\ldots,x_k)=l\}
  \end{align*}
  with differential $\der: \MC_{*,*}(G)\to \MC_{*-1,*}(G)$ defined by
  \begin{align*}
  \der (x_0,\ldots,x_k) = \sum_{1\le i\le k-1} (-1)^i \der_i (x_0,\ldots,x_k)
  \end{align*}
  where
  \begin{align*}
  &\der_i (x_0,\ldots,x_k) \\
  &= \left\{\begin{array}{ll}(x_0,\ldots,\hat x_i,\ldots,x_k)& \text{if }\ell(x_0,\ldots,\hat x_i,\ldots,x_k) = \ell(x_0,\ldots,x_k),\\ 0 & \text{otherwise.}\end{array}\right.
  \end{align*}
  The magnitude homology $\MH_{*,*}(G)$ is defined as
  \begin{align*}
  \MH_{k,l}(G) = H_k(\MC_{*,l}(G)).
  \end{align*}
  \end{defn}
  
  Hepworth and Willerton proved many interesting properties of magnitude homology. In this paper, we perform computations starting from the magnitude chain complex, so we do not need most of the properties.
  
  We recall the following definition.
  \begin{defn}[Diagonal graphs]
  A diagonal graph is a graph whose magnitude homology is diagonal, i.e., $\MH_{k,l}(G)\ne 0$ only if $k=l$.
  \end{defn}
  Diagonal graphs are interesting because their magnitude homology is completely determined by magnitude (\cite{HW17}, Proposition 34). In contrast, graphs in general can have the same magnitude but different magnitude homology, as shown in Appendix \ref{SecMHvsMag}.
  
  Hepworth and Willerton proved that trees are diagonal (op.~cit., Corollary 31) and that joins of (non-empty) graphs are diagonal (op.~cit., Theorem 37). Recall that the join $G\star H$ of two graphs $G$ and $H$ has vertex set $V(G\star H) = V(G) \coprod V(H)$ and edge set $E(G\star H) = E(G) \coprod E(H) \coprod (V(G) \t V(H))$.
  
  \subsection{Algebraic Morse theory}
  This part follows Sk\"oldberg \cite{Sko06} and Lampret and Vavpeti\v{c} \cite{LV16}.
  Fix a commutative ring $R$. (We use $R = \bZ$ throughout this paper.)
  Let $C_*$ be a chain complex of $R$-modules $$\cdots \to C_{k+1} \xrw{\der_{k+1}} C_k \xrw{\der_k} C_{k-1} \to \cdots.$$
  Assume that for each $k\in \bZ$, we have a direct sum decomposition $C_k = \bop_{i\in I_k} C_{k,i}$ where $I_k$ is the index set.
  For our purpose, we assume that $I_k$ is finite and each $C_{k,i}$ is isomorphic to $R$.
  For $i\in I_k$ and $j\in I_{k-1}$, let $\der_{k,i,j}$ denote the composition $$C_{k,i} \inj C_k \xrw{\der_k} C_{k-1} \sur C_{k-1, j},$$ where the first map is the obvious inclusion, and the third map is the obvious projection.
  We assume that each $\der_{k,i,j}$ is either $0$ or an isomorphism.
  Define $\Ga_{C_*}$ to be the directed graph with vertex set $\coprod_{k\in \bZ} I_k$ and edge set $$\coprod_{k\in \bZ} \{(i,j):i\in I_k, j\in I_{k-1},\der_{k,i,j}\ne 0\}.$$
  \begin{defn}[Morse matching]
  Let $M$ be a (not necessarily perfect) matching of $\Ga_{C_*}$.
  Let $\Ga_{C_*}^M$ denote the graph $\Ga_{C_*}$ with edges in $M$ having reversed direction.
  Then $M$ is called a Morse matching if $\Ga_{C_*}^M$ is acyclic.
  \end{defn}
  \begin{rmk}\label{RmkMorseCycle}
  If there exists a cycle in $\Ga_{C_*}^M$, then it must be of the form 
  $$\begin{tikzcd}[column sep=small, row sep=small]
  a_1\arrow{d} & a_2 \arrow{d} & \cdots \arrow{d} & a_p \arrow{d} & a_{p+1} = a_1\\
  b_1\arrow{ru} & b_2 \arrow{ru} & \cdots \arrow{ru} & b_p \arrow{ru} &
  \end{tikzcd}$$
  where $a_i\in I_k$, $b_i\in I_{k-1}$ for all $i$, for some fixed $k$. (Vertex $b_i$ and $a_{i+1}$ are matched.)
  \end{rmk}
  
  Given a Morse matching $M$, we can find a smaller chain complex homotopy equivalent to $C_*$.
  Let $I^\circ_k$ denote the set of vertices in $I_k$ unmatched in $M$.
  We define a chain complex $C^\circ_*$ with $C^\circ_k = \bop_{i\in I^\circ_k} C_{k,i}$.
  Let us describe the differential $\der^\circ: C^\circ_* \to C^\circ_{*-1}$.
  For $u\in I^\circ_k$, $v\in I^\circ_{k-1}$, let 
  $\Ga_{u,v}^M$ denote the set of paths
  $$\ga = (u=v_1\to v_2 \to \cdots \to v_{2r}=v)$$ in 
  $\Ga_{C_*}^M$ with 
  $v_{2i-1} \in I_k$, $v_{2i}\in I_{k-1}$ for all $i$.
  For such a path, we define $\der^\circ_\ga: C_{k,u}\to C_{k-1,v}$ as
  \begin{align*}
  \der^\circ_\ga = (-1)^{r-1} \der_{k,v_{2r-1},v_{2r}} \der^{-1}_{k,v_{2r-1},v_{2r-2}} \cdots \der_{k,v_3,v_4}\der^{-1}_{k,v_3,v_2} \der_{k,v_1,v_2}.
  \end{align*}
  Then the differential $\der^\circ$ restricted on $C_{k,u}$, $u\in I^\circ_k$ is defined as
  \begin{align*}
  \der^\circ|_{C_{k,u}} = \sum_{v\in I_{k-1}^\circ, \ga \in \Ga_{u,v}^M} \der^\circ_\ga.
  \end{align*}
  This determines the differential $\der^\circ: C^\circ_*\to C^\circ_{*-1}$.
  It turns out that $(C^\circ_*,\der^\circ)$ is a chain complex, and furthermore, we have the following theorem.
  \begin{thm}\label{ThmMorse}
  The chain complex $(C^\circ_*, \der^\circ)$ is homotopy equivalent to $(C_*, \der)$.
  \end{thm}
  \section{Description of matchings}\label{SecDes}
  Morse matchings are useful for simplifying a chain complex. However, it can sometimes be cumbersome to describe a Morse matching and to prove its correctness. In this section we study special kinds of Morse matchings for magnitude chain complexes that are easier to deal with.
  
  Strictly speaking, the magnitude chain complex is not a single chain complex, but one chain complex for each $l$. Nevertheless, we treat these chain complexes uniformly. By a Morse matching of $\MC_{*,*}(G)$, we mean a Morse matching of $\MC_{*,l}(G)$ for each $l$.
  
  The index set we use is $$I_{k,l}(G) = \{(x_0,\ldots,x_k)\in V(G)^{k+1}, x_i\ne x_{i+1}\forall i, \ell(x_0,\ldots,x_k)=l\}.$$
  In the following, by ``a sequence $(x_0,\ldots,x_k)$'' we mean a sequence $(x_0,\ldots,x_k)\in I_{k,l}(G)$.
  
  First we introduce the notion of matching states.
  \begin{notn}[Matching state]
  Fix a (not necessarily Morse) matching of $\MC_{*,*}(G)$.
  The matching state of a sequence $(x_0,\ldots,x_k)$ is one of the following:
  \begin{enumerate}
  \item unmatched, if it is not matched to another sequence;
  \item insert($i,v$), if it is matched to the sequence $(x_0,\ldots,x_i,v,x_{i+1},\ldots,x_k)$;
  \item delete($i$), if it is matched to the sequence $(x_0,\ldots,\hat x_i,\ldots,x_k)$.
  \end{enumerate}
  \end{notn}
  
  In this paper, all the matchings we construct satisfy the property that (roughly speaking) the matching state of a matched sequence is determined by a prefix of it. Therefore we make the following definition.
  \begin{defn}[Prefix matching]
  A prefix matching of $\MC_{*,*}(G)$ is a matching satisfying the following properties.
  \begin{enumerate}
  \item If $(x_0,\ldots,x_k)$ has matching state insert($i, v$), then any sequence of the form $(x_0,\ldots,x_{i+1},y_{i+2},\ldots,y_{k^\p})$ has matching state insert($i, v$).
  \item If $(x_0,\ldots,x_k)$ has matching state delete($i$), then any sequence of the form $(x_0,\ldots,x_{i+1},y_{i+2},\ldots,y_{k^\p})$ has matching state delete($i$).
  \end{enumerate}
  In other words, for a sequence $(x_0,\ldots,x_k)$, the prefix $(x_0,\ldots,x_{i+1})$ determines whether the sequence has matching state insert($i, v$), delete($i$), or neither.
  \end{defn}
  
  In order to make it easy to check the correctness of the description of a prefix matching, we make the following definition.
  \begin{defn}[Matching rule]\label{DefnMatRule}
  A matching rule is a function $F$ that maps a finite sequence $(x_0,\ldots,x_k)$ of vertices to one of the following symbols:
  \begin{enumerate}
  \item $\ep$, meaning ``idle'';
  \item $\iota(v)$ for some $v\in V(G)$, meaning insert($*$, $v$).
  \item $\de$, meaning delete($*$);
  \end{enumerate}
  
  Let $M$ be a prefix matching. We say $M$ is the prefix matching generated by $F$ if for any sequence $(x_0,\ldots,x_k)$ with unmatched prefix $(x_0,\ldots,x_{k-1})$, the matching state of $(x_0,\ldots,x_k)$ is
  \begin{enumerate}
  \item insert($k-1$, $v$) if and only if $F(x_0,\ldots,x_k) = \iota(v)$;
  \item delete($k-1$) if and only if $F(x_0,\ldots,x_k) = \de$.
  \end{enumerate}
  
  Note that there does not always exist a prefix matching generated by $F$, because we have not put any restrictions on $F$.
  
  We say a matching rule is valid if it satisfies the following properties.
  \begin{enumerate}
  \item If $F(x_0,\ldots,x_k) = \iota(v)$, then $d(x_{k-1},v)+d(v,x_k)=d(x_{k-1},x_k)$, $F(x_0,\ldots,x_{k-1}, v) = \ep$ and $F(x_0,\ldots,x_{k-1}, v, x_k) = \de$.
  \item If $F(x_0,\ldots,x_k) = \de$, then $d(x_{k-2},x_{k-1})+d(x_{k-1},x_k)=d(x_{k-2},x_k)$, and $F(x_0,\ldots,x_{k-2},x_k) = \iota(x_{k-1})$.
  \end{enumerate}
  \end{defn}
  \begin{rmk}
  It is easy to see that in the definition of matching rules, value of $F(x_0,\ldots,x_k)$ is meaningful only when $(x_0,\ldots,x_{k-1})$ is unmatched. So we can assume that $F(x_0,\ldots,x_k)=\ep$ if $(x_0,\ldots,x_j)$ is matched for some $j<k$.
  
  For simplicity, when we describe a matching rule, we usually omit sequences on which its value is $\ep$. However, we never omit any sequences on which the value is $\de$ or $\iota(*)$.
  \end{rmk}
  \begin{lemma}\label{LemmaValidMatRule}
  If $F$ is a valid matching rule, then there exists a prefix matching generated by $F$.
  \end{lemma}
  \begin{proof}
  Fix a sequence $(x_0,\ldots,x_k)$. Let $j$ be the smallest number such that $F(x_0,\ldots,x_{j}) \ne \ep$.
  If such $j$ does not exist, then $(x_0,\ldots,x_k)$ is unmatched.
  
  If $F(x_0,\ldots,x_{j})=\iota(v)$, then $(x_0,\ldots,x_k)$ is matched to $(x_0,\ldots,x_{j-1}, v, x_{j}, \ldots,x_k)$.
  By valid property (1) in Definition \ref{DefnMatRule}, we have $F(x_0,\ldots,x_{j-1}, v) = \ep$ and $F(x_0,\ldots,x_{j-1}, v, x_{j}) = \de$, so $(x_0,\ldots,x_{j-1}, v, x_{j},\ldots,x_k)$ is matched to $(x_0,\ldots,x_k)$.
  
  If $F(x_0,\ldots,x_{j}) = \de$, then $(x_0,\ldots,x_k)$ is matched to $(x_0,\ldots,x_{j-2}, x_{j}, \ldots, x_k)$.
  By valid property (2) in Definition \ref{DefnMatRule}, we have $F(x_0,\ldots,x_{j-2}, x_{j}) = \iota(x_{j-1})$, so $(x_0,\ldots,x_{j-2}, x_{j},\ldots,x_k)$ is matched to $(x_0,\ldots,x_k)$.
  
  Therefore under the matching rule, every sequence is matched to at most one other sequence. So this gives a valid prefix matching.
  \end{proof}
  
  With Lemma \ref{LemmaValidMatRule}, we can verify that a matching rule generates a prefix matching. In order to apply to algebraic Morse theory, we would like to know when a matching rule generates a Morse matching.
  It turns out that this problem is not easy.
  \begin{eg}
  There exist prefix matchings that are not Morse matchings.
  \begin{figure}[h]
  \includegraphics[scale=0.5]{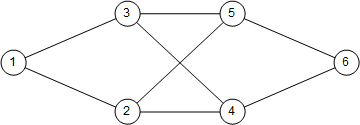}
  \caption{Example graph for a non-Morse matching}
  \label{FigNonMorse}
  \end{figure}
  Consider the graph in Figure \ref{FigNonMorse} and the following matching rule.
  \begin{enumerate}
  \item $F(1, 4) = \iota(2)$, $F(1, 2, 4) = \de$;
  \item $F(1, 5) = \iota(3)$, $F(1, 3, 5) = \de$;
  \item $F(1, 2, 6) = \iota(5)$, $F(1, 2, 5, 6) = \de$;
  \item $F(1, 3, 6) = \iota(4)$, $F(1, 3, 4, 6) = \de$.
  \end{enumerate}
  One can verify that this is a valid matching rule, and therefore generates a prefix matching.
  However, this is not a Morse matching, because we have the following zig-zag cycle.
  $$\begin{tikzcd}[column sep=0em]
  (1, 2, 4, 6)\arrow{d} & (1, 2, 5, 6) \arrow{d} & (1, 3, 5, 6) \arrow{d} & (1, 3, 4, 6) \arrow{d} & (1, 2, 4, 6)\\
  (1, 2, 6)\arrow{ru} & (1, 5, 6) \arrow{ru} & (1, 3, 6) \arrow{ru} & (1, 4, 6) \arrow{ru} &
  \end{tikzcd}$$
  \end{eg}
  
  Nevertheless, there are some general facts that can simplify Morse-ness proofs.
  \begin{lemma}\label{LemmaNonMorse}
  Fix a prefix matching $M$. Suppose in the graph $\Ga_{\MC_{*,*}(G)}^M$, there is a directed cycle
  $$\begin{tikzcd}[column sep=small, row sep=small]
  a_1\arrow{d} & a_2 \arrow{d} & \cdots \arrow{d} & a_p \arrow{d} & a_{p+1} = a_1\\
  b_1\arrow{ru} & b_2 \arrow{ru} & \cdots \arrow{ru} & b_p \arrow{ru} &
  \end{tikzcd}$$
  (Recall Remark \ref{RmkMorseCycle}.)
  
  Say $a_i$ transforms into $b_i$ by delete($d_i$), and $b_i$ transforms into $a_{i+1}$ by insert($c_i$, $u_i$).\footnote{We abuse notation from matching states. Note that $a_i$ has matching state delete($c_{i-1}+1$) and $b_i$ has matching state insert($c_i$, $u_i$).}
  Then we have $d_{i+1} \ne c_i+1$, $d_{i+1}\le c_i+2$, and $d_i\le c_i+1$.
  
  (Because this is a cycle, all indices are mod $p$. That is, $a_{p+i}=a_i$, $b_{p+i}=b_i$, $c_{p+i}=c_i$, $d_{p+i}=d_i$.)
  \end{lemma}
  \begin{proof}
  Write $a_i = (a_{i,0},\ldots,a_{i,k})$ and $b_i = (b_{i,0},\ldots,b_{i,k-1})$.
  
  (1) $d_{i+1} \ne c_i+1$. If $d_{i+1}=c_i+1$, then $b_i = b_{i+1}$, which cannot happen because there is at most one direct edge between two sequences in $\Ga_{\MC_{*,*}(G)}^M$.
  
  (2) $d_{i+1} \le c_i+2$. If $d_{i+1} \ge c_i+3$, then $(b_{i+1,0},\ldots,b_{i+1,c_i+2}) = (a_{i+1,0},\ldots,a_{c_i+2})$.
  Because of the edge $b_i\to a_{i+1}$, we know $a_{i+1}$ has matching state delete($c_i+1$). By definition of prefix matchings, sequence $b_{i+1}$ also has matching state delete($c_i+1$). Then there cannot be an edge $b_{i+1} \to a_{i+2}$. Contradiction.
  
  (3) $d_i \le c_i+1$. If $d_i \ge c_i+2$, then $(b_{i,0},\ldots,b_{i,c_i+1}) = (a_{i,0},\ldots,a_{i,c_i+1})$.
  Because of the edge $b_i \to a_{i+1}$, we know $b_i$ has matching state insert($c_i$, $u_i$). By definition of prefix matchings, sequence $a_i$ also has matching state insert($c_i$, $u_i$). Then there cannot be an edge $b_{i-1}\to a_i$. Contradiction.
  \end{proof}
  
  Some of the matchings we construct are for proving that a graph is diagonal. Therefore we make the following definition.
  \begin{defn}[Diagonal matching rule]
  Let $F$ be a valid matching rule. We say $F$ is a diagonal matching rule if for any sequence $(x_0,\ldots,x_k)$ with unmatched prefix $(x_0,\ldots,x_{k-1})$ and $d(x_{k-1},x_k)\ge 2$, we have $F(x_0,\ldots,x_k)\ne \ep$.
  \end{defn}
  \begin{lemma}\label{LemmaDiagMatRuleProperty}
  Let $F$ be a diagonal matching rule and $M$ be the prefix matching generated by $F$.
  If $(x_0,\ldots,x_k)$ is unmatched, then $d(x_i,x_{i+1})=1$ for all $i$.
  If $F(x_0,\ldots,x_k) = \iota(v)$, then we have $d(x_{k-1},v) = 1$.
  \end{lemma}
  \begin{proof}
  The second assertion follows from the first by valid property (1) in Definition \ref{DefnMatRule}. So we only need to prove the first assertion.
  Suppose $(x_0,\ldots,x_k)$ is unmatched and $d(x_j,x_{j+1})\ge 2$ for some $j<k$.
  Because it is unmatched, the prefix $(x_0,\ldots,x_j)$ is unmatched.
  Then by definition of diagonal matchings, $F(x_0,\ldots,x_{j+1})\ne \ep$. Contradiction.
  \end{proof}
  \begin{coro}\label{CoroDiagMatRuleImplyGraph}
  Let $F$ be a diagonal matching rule and $M$ be the prefix matching generated by $F$.
  If $M$ is a Morse matching, then the graph $G$ is diagonal.
  \end{coro}
  \begin{proof}
  By Lemma \ref{LemmaDiagMatRuleProperty}, all unmatched sequences have $\ell(x_0,\ldots,x_k)=k$.
  So $I_{k,l}^\circ$ is nonempty only when $k=l$ and therefore $\MC^\circ_{k,l}(G)$ is nontrivial only when $k=l$.
  By Theorem \ref{ThmMorse}, magnitude homology $\MH_{*,*}(G)$ is homology of $\MC^\circ_{*,*}(G)$ with certain differentials.
  So $\MH_{k,l}(G)$ is nontrivial only when $k=l$.
  \end{proof}
  \begin{lemma}\label{LemmaNonMorseDiag}
  Work in the setting of Lemma \ref{LemmaNonMorse} and assume in addition that $M$ is generated by a diagonal matching rule $F$.
  Then $d_i = c_i+1$ for all $i$.
  \end{lemma}
  \begin{proof}
  For a sequence $x_0,\ldots,x_k$, let $\rho(x_0,\ldots,x_k)$ denote the largest integer $j$ such that $\ell(x_0,\ldots,x_j) = j$.

  By Lemma \ref{LemmaNonMorse}, we have $d_i \le c_i+1$.
  Because we can delete $a_{i,d_i}$ from $a_i$, we have $d(a_{i,d_i-1}, a_{i,d_i+1})\ge 2$.
  So $\rho(b_i) \le d_i-1$.
  By Lemma \ref{LemmaDiagMatRuleProperty}, the prefix $(b_{i,0},\ldots,b_{i,\rho(b_i)+1})$ is matched.
  By definition of prefix matchings, sequence $b_i$ has the same matching state as its prefix $(b_{i,0},\ldots,b_{i,\rho(b_i)+1})$.
  So we must have $c_i \le \rho(b_i)$.
  Combining these, we get $c_i = \rho(b_i) = d_i-1$.
  \end{proof}
  \begin{lemma}\label{LemmaNonMorseDC}
  Work in the setting of Lemma \ref{LemmaNonMorse} and assume in addition that $d_i = c_i+1$ and $d(a_{i+1,c_i},a_{i+1,c_i+1})=1$ for all $i$. Then $d(a_{i,d_i-1}, a_{i,d_i}) = 1$ for all $i$.
  \end{lemma}
  \begin{proof}
  For a sequence $(x_0,\ldots,x_k)$, define $\tau(x_0,\ldots,x_k)$ be the sequence $$(\ell(x_0,x_1), \ell(x_0,x_1,x_2),\ldots,\ell(x_0,\ldots,x_k)).$$
  Because $d_i = c_i+1$ and $d(a_{i+1,c_i},a_{i+1,c_i+1})=1$, we have $\tau(a_{i+1})\le \tau(a_i)$ (where $\le$ is lexicographical order), and equality holds if and only if $d(a_{i,d_i-1}, a_{i,d_i}) = 1$.
  Because $\tau(a_1) \ge \cdots \ge \tau(a_p) \ge \tau(a_{p+1})=\tau(a_1)$, we see all inequality signs are equalities, and therefore $d(a_{i,d_i-1},a_{i,d_i})=1$ for all $i$.
  \end{proof}
  \begin{coro}\label{CoroNonMorseDiagDA}
  Work in the setting of Lemma \ref{LemmaNonMorse} and assume in addition that $M$ is generated by a diagonal matching rule $F$. Then $d(a_{i,d_i-1}, a_{i,d_i}) = 1$ for all $i$.
  \end{coro}
  \begin{proof}
  By Lemma \ref{LemmaDiagMatRuleProperty}, Lemma \ref{LemmaNonMorseDiag}, and Lemma \ref{LemmaNonMorseDC}.
  \end{proof}
  \section{Computations}\label{SecComp}
  In this section we perform the computations. Each computation is in roughly four steps.
  \begin{enumerate}
  \item Setup notions and describe the matching rule $F$.
  \item Prove that $F$ is valid, and sometimes that $F$ is diagonal.
  \item Prove that $F$ generates a Morse matching.
  \item Analyze unmatched sequences, and analyze differentials in $\MC_{*,*}^\circ(G)$ if there are any.
  \end{enumerate}
  Usually step (2) and step (4) are routine, and step (3) is the most difficult.
  Sometimes step (3) involves analyzing unmatched sequences.
  \subsection{Warmup: Trees}
  Hepworth and Willerton proved that trees are diagonal (\cite{HW17}, Corollary 31). Their proof relies on a Mayer-Vietoris theorem, which can take some effort to prove. As a warmup for more complicated computations, we give another proof of this fact using algebraic Morse theory.
  \begin{prop}\label{PropTree}
  Trees are diagonal.
  \end{prop}
  \begin{proof}
  Let $G$ be a tree.
  We define a function $\sm: \{(u,v) \in V(G) \t V(G) : d(u,v)\ge 2\} \to V(G)$ that maps $(u,v)$ to the unique vertex $w$ with $d(u,w)=1$ and $d(u,v) = d(u,w)+d(v,w)$.
  Existence and uniqueness follows from that $G$ is a tree.
  
  Let us describe the matching rule.
  Fix a sequence $(x_0,\ldots,x_k)$ with unmatched prefix $(x_0,\ldots,x_{k-1})$.
  \begin{enumerate}
  \item If $k\ge 2$ and $x_{k-1} = \sm(x_{k-2},x_k)$, then $F(x_0,\ldots,x_k) = \de$.
  \item If $k\ge 1$, $d(x_{k-1}, x_k)\ge 2$, and not $(k\ge 2 \land x_{k-1} = \sm(x_{k-2},x_k))$, then $F(x_0,\ldots,x_k) = \iota(\sm(x_{k-1},x_k))$.
  \end{enumerate}
  
  Let us prove that $F$ is a valid matching rule. 
  Fix a sequence $(x_0,\ldots,x_k)$ with unmatched prefix $(x_0,\ldots,x_{k-1})$.
  \begin{enumerate}
  \item Suppose $k\ge 2$ and $x_{k-1} = \sm(x_{k-2},x_k)$. Clearly $d(x_{k-2},x_k)\ge 2$. Also, if $k\ge 3$ and $x_{k-2} = \sm(x_{k-3},x_k)$, then $x_{k-2} = \sm(x_{k-3}, x_{k-1})$ and $F(x_0,\ldots,x_{k-1}) = \de$, which is not true. So we have $\lnot (k\ge 3 \land x_{k-2} = \sm(x_{k-3},x_k))$.
  So $F(x_0,\ldots,x_{k-2},x_k) = \iota(\sm(x_{k-2},x_k)) = \iota(x_{k-1})$.
  \item Suppose $k\ge 1$, $d(x_{k-1}, x_k)\ge 2$, and not $(k\ge 2 \land x_{k-1} = \sm(x_{k-2},x_k))$.
  Because $d(x_{k-1},\sm(x_{k-1},x_k))=1$, we have $F(x_0,\ldots,x_{k-1},\sm(x_{k-1},x_k))\ne \iota(*)$.
  If $F(x_0,\ldots,x_{k-1},\sm(x_{k-1},x_k)) = \de$, then $k\ge 2$ and $x_{k-1} = \sm(x_{k-2}, \sm(x_{k-1},x_k))$, and therefore $x_{k-1} = \sm(x_{k-2}, x_k)$, which is not true.
  So $F(x_0,\ldots,x_{k-1},\sm(x_{k-1},x_k)) = \ep$ and thus $F(x_0,\ldots,x_{k-1},\sm(x_{k-1},x_k), x_k) = \de$.
  \end{enumerate}
  
  It is easy to see that $F$ is a diagonal matching rule.
  
  Let $M$ be the prefix matching generated by $F$. Let us prove that $M$ is a Morse matching. Work in the setting of Lemma \ref{LemmaNonMorse}. Suppose in $\Ga_{\MC_{*,*}(G)}^M$ there is a directed cycle $a_1 \to b_1 \to \cdots \to b_p \to a_{p+1}=a_1$. 
  Define $d_i$, $c_i$, $u_i$ accordingly.
  
  By Corollary \ref{CoroNonMorseDiagDA}, we have $d(a_{1,d_1-1}, a_{1,d_1}) = 1$.
  Because of the edge $a_1 \to b_1$, we have $d(a_{1,d_1-1}, a_{1,d_1}) + d(a_{1,d_1}, a_{1,d_1+1}) = d(a_{1,d_1-1}, a_{1,d_1+1})$.
  So $a_{1,d_1} = \sm(a_{1,d_1-1}, a_{1,d_1+1})$.
  Also, by the edge $b_1 \to a_2$, we have $a_{2,c_1+1} = \sm(a_{2,c_1}, a_{2,c_1+2})$.
  By Lemma \ref{LemmaNonMorseDiag}, we have $c_1 = d_1-1$.
  Also, we have $a_{1,j} = b_{1,j} = a_{2,j}$ for $j\le d_1-1$, and $a_{1,j} = b_{1,j-1} = a_{2,j}$ for $j\ge d_1+1$.
  So $a_{2,d_1} = a_{1,d_1}$, and therefore $a_2 = a_1$.
  Then edges $a_1 \to b_1$ and $b_1\to a_2$ cannot both exist in $\Ga_{\MC_{*,*}(G)}^M$. Contradiction.
  So $M$ is a Morse matching.
  
  By Corollary \ref{CoroDiagMatRuleImplyGraph}, the graph $G$ is diagonal. Actually, the unmatched sequences are $(v)$ for $v\in V(G)$ and $(u,v,u,\cdots, u \text{ or } v)$ for $(u,v)$ or $(v,u) \in E(G)$.
  Their homology classes form a basis of $\MH_{*,*}(G)$.
  \end{proof}
  
  In fact, the computation for trees can be directly applied to graphs that are both ptolemaic and geodetic. However, it turns out that this does not give previously-unknown diagonal graphs. See Appendix \ref{SecGeoPto}.
  
  \subsection{A new class of diagonal graphs}
  Hepworth and Willerton proved that joins of graphs are diagonal (\cite{HW17}, Theorem 37). We give a new class of diagonal graphs, containing all joins.
  \begin{defn}[Pawful graphs]\label{DefnPawful}
  A pawful graph is a connected graph of diameter at most $2$ satisfying the property that for any three vertices $u,v,w$ with $d(u,v)=d(v,w)=2$ and $d(u,w)=1$, there exists a vertex $x$ such that $d(x,u)=d(x,v)=d(x,w)=1$.\footnote{They are so named because the subgraph induced by vertices $u,v,w,x$ is a paw.}
  \end{defn}
  \begin{eg}
  Joins of graphs are pawful. Suppose we have a join $G\star H$ with $G$ and $H$ non-empty. Then $G\star H$ is connected and has diameter at most $2$.
  If there are three vertices $u,v,w$ with $d(u,v)=d(v,w)=2$ and $d(u,w)=1$, then $u,v,w$ must be all in $G$ or all in $H$, and we can take $x$ to be an arbitrary vertex of the other side.
  
  There exist pawful graphs that are not joins, e.g.~the complement of $C_n$ for $n\ge 6$. So the class of pawful graphs is strictly larger than the class of joins.
  \end{eg}
  \begin{thm}\label{ThmPawful}
  Pawful graphs are diagonal.
  \end{thm}
  \begin{proof}
  Let $G$ be a pawful graph.
  We choose a function $f : \{(u,v)\in V(G)\t V(G) : d(u,v)=2\} \to V(G)$ that maps $(u,v)$ to any vertex $w$ with $d(u,w)=d(v,w)=1$, and a function $g: \{(u,v,w)\in V(G)^{\t 3} : d(u,v)=d(v,w)=2, d(u,w)=1\} \to V(G)$ that maps $(u,v,w)$ to any vertex $x$ with $d(u,x)=d(v,x)=d(w,x)=1$.
  Because $G$ is pawful, such functions $f$ and $g$ exist.

Let us describe the matching rule.
  Fix a sequence $(x_0,\ldots,x_k)$ with unmatched prefix $(x_0,\ldots,x_{k-1})$.
  \begin{enumerate}
  \item If $k=1$ and $d(x_0,x_1)=2$, then $F(x_0,x_1) = \iota(f(x_0,x_1))$.
  \item If $k\ge 2$, $d(x_{k-1},x_k)=2$, and $d(x_{k-2},x_k)=1$, then $F(x_0,\ldots,x_k)=\iota(x_{k-2})$.
  \item If $k\ge 2$, $d(x_{k-1},x_k)=2$, and $d(x_{k-2},x_k)=2$, then
  $F(x_0,\ldots,x_k) = \iota(g(x_{k-2},x_k,x_{k-1}))$.
  \item If $k=2$, $d(x_0,x_2)=2$ and $x_1 = f(x_0,x_2)$, then $F(x_0,x_1,x_2)=\de$.
  \item If $k\ge 3$, $d(x_{k-2},x_k)=2$, $d(x_{k-3}, x_k)=1$, and $x_{k-3}=x_{k-1}$, then $F(x_0,\ldots,x_k) = \de$.
  \item If $k\ge 3$, $d(x_{k-2},x_k)=2$, $d(x_{k-3}, x_k)=2$, and $x_{k-1} = g(x_{k-3}, x_k, x_{k-2})$, then $F(x_0,\ldots,x_k) = \de$.
  \end{enumerate}
  
  Let us prove that $F$ is a valid matching rule. Note that $F$ is a diagonal matching rule.\footnote{Strictly speaking, we perform induction on $\ell$ and prove at the same time that $F$ is valid and diagonal up to sequences with $\ell(x_0,\ldots,x_k)\le \ell$. For simplicity, we use the fact that $F$ is diagonal during the proof that $F$ is valid. Because we are doing induction, this is not circular argument.}
  \begin{enumerate}
  \item Suppose $k=1$ and $d(x_0,x_1)=2$. Then $d(x_0,f(x_0,x_1))=1$, so $F(x_0,f(x_0,x_1)) = \ep$. Then we have $F(x_0,f(x_0,x_1),x_1) = \de$ by rule (4).
  \item Suppose $k\ge 2$, $d(x_{k-1},x_k)=2$, and $d(x_{k-2},x_k)=1$.
  We have $d(x_{k-2},x_{k-2})=0$, so $F(x_0,\ldots,x_{k-1},x_{k-2}) \ne \de$.
  Because $(x_0,\ldots,x_{k-1})$ is unmatched, we have $d(x_{k-2},x_{k-1})=1$.
  So $F(x_0,\ldots,x_{k-1},x_{k-2}) \ne \iota(*)$ and therefore $F(x_0,\ldots,x_{k-1},x_{k-2})=\ep$.
  Then we have $F(x_0,\ldots,x_{k-1},x_{k-2},x_k)=\de$ by rule (5).
  \item Suppose $k\ge 2$, $d(x_{k-1},x_k)=2$, and $d(x_{k-2},x_k)=2$.
  We have $d(x_{k-2}, g(x_{k-2},x_k,x_{k-1})) = d(x_{k-1}, g(x_{k-2},x_k,x_{k-1})) = 1$, so $F(x_0,\ldots,x_{k-1}, g(x_{k-2},x_k,x_{k-1})) = \ep$.
  Then we have $F(x_0,\ldots,x_{k-1},g(x_{k-2},x_k,x_{k-1}),x_k)=\ep$ by rule (6).
  \item Suppose $k=2$, $d(x_0,x_2)=2$ and $x_1 = f(x_0,x_2)$. Then $F(x_0,x_2)=\iota(x_1)$ by rule (1).
  \item Suppose $k\ge 3$, $d(x_{k-2},x_k)=2$, $d(x_{k-3}, x_k)=1$, and $x_{k-3}=x_{k-1}$. Then $F(x_0,\ldots,x_{k-2},x_k) = \iota(x_{k-1})$ by rule (2). 
  \item Suppose $k\ge 3$, $d(x_{k-2},x_k)=2$, $d(x_{k-3}, x_k)=2$, and $x_{k-1} = g(x_{k-3}, x_k, x_{k-2})$. Then $F(x_0,\ldots,x_{k-2}, x_k) = \iota(x_{k-1})$ by rule (3).
  \end{enumerate}
  
  Let $M$ be the prefix matching generated by $F$. Let us prove that $M$ is a Morse matching.
  Work in the setting of Lemma \ref{LemmaNonMorse}.
  Suppose in $\Ga_{\MC_{*,*}(G)}^M$ there is a directed cycle $a_1 \to b_1 \to \cdots \to b_p \to a_{p+1} = a_1$, Define $d_i$, $c_i$, $u_i$ accordingly.
  By rotating, we can assume WLOG that $c_1$ is the smallest $c_i$ among all $i$'s.
  
  By Lemma \ref{LemmaNonMorse}, we have $d_2 \le c_1+2$ and $d_2 \ne c_1+1$.
  So $d_2 = c_1+2$ and therefore $c_2 = c_1+1$ by Lemma \ref{LemmaNonMorseDiag}.
  By definition of $F$, we have $d(a_{3,c_1}, a_{3,c_1+2}) \le 1$.
  Actually, for all $i\ge 3$, if $d_i = c_1+2$, then $d(a_{i+1,c_1}, a_{i+1,c_1+2}) \le 1$ by definition of $F$; if $d_i \ne c_1+2$, then $d(a_{i+1,c_1}, a_{i+1,c_1+2}) = d(a_{i,c_1}, _{i,c_1+2}) \le 1$. (We know that $a_{i,c_1} = a_{i+1,c_1}$ because $d_i = c_i+1 \ge c_1+1$.)
  So we have $d(a_{i,c_1}, a_{i,c_1+2}) \le 1$ for all $i\ge 3$.
  
  For all $i$, if $d(a_{i,c_1}, a_{i,c_1+2}) \le 1$, then $d_i \ne c_1+1$ and therefore $c_i\ne c_1$.
  So $c_i\ne c_1$ for all $i\ge 3$. However, we know that $c_{p+1} = c_1$. Contradiction.
  So $M$ is a Morse matching.
  
  By Corollary \ref{CoroDiagMatRuleImplyGraph}, the graph $G$ is diagonal.
  \end{proof}
  \subsection{Icosahedral graph}
  Hepworth and Willerton \cite{HW17} did computer computations and conjectured that the icosahedral graph is diagonal. We prove this using algebraic Morse theory.
  \begin{figure}[h]
  \includegraphics[scale=0.5]{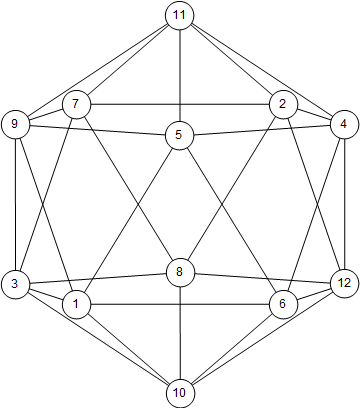}
  \caption{Icosahedral graph}
  \label{FigIco}
  \end{figure}
  \begin{thm}\label{ThmIco}
  The icosahedral graph is diagonal.
  \end{thm}
  \begin{proof}
  Let $G$ be the icosahedral graph. We fix an embedding of $G$ as the skeleton of an icosahedron in a three dimensional Euclidean space. See Figure \ref{FigIco}. Vertices $2$, $7$, $8$ are in the back and invisible, and the other vertices are in the front and visible.
  
  We choose a function $f: V(G) \to V(G)$ that maps a vertex $u$ to any vertex $v$ with $d(u,v)=1$.
  We define two functions $g_L, g_R: \{(u,v)\in V(G)\t V(G) : d(u,v) = 2\} \to V(G)$ as follows.
  Rotate (without reflection) the icosahedron so that $u$ is at the position of vertex $10$ in Figure \ref{FigIco}, and $v$ is at the position of vertex $5$. Then $g_L(u,v)$ is the vertex at the position of vertex $1$, and $g_R(u,v)$ is the vertex at the position of vertex $6$.
  It is easy to see that we have $d(u,g_L(u,v)) = d(v,g_L(u,v)) = d(u,g_R(u,v)) = d(v,g_R(u,v)) = 1$, $g_L(u,v) \ne g_R(u,v)$, and $g_L(u,v) = g_R(v,u)$.
  
  We define a function $\xi: \{(u,v,w) \in V(G)^{\t 3} : d(u,v)=1, d(v,w)=2\} \to V(G)$ that maps $(u,v,w)$ to the $x \in \{g_L(v,w), g_R(v,w)\}$ with smaller $d(x,u)$. (If there is tie, choose any.) Note that when $d(u,w)\ne 3$, there is no tie, and $d(u,\xi(u,v,w))\le 1$.
  We define a function $\ze: \{(u,v,w,x) \in V(G)^{\t 4} : d(u,v)=1, d(w,x)=2, d(v,x)=3, u\ne w\}$ that maps $(u,v,w,x)$ to the unique vertex $y\in \{g_L(w,x), g_R(w,x)\}$ with smaller $d(y, u)$.
  By checking all possible relative positions one can verify that $\ze$ is well-defined.
  
  Let us describe the matching rule. Fix a sequence $(x_0,\ldots,x_k)$ with unmatched prefix $(x_0,\ldots,x_{k-1})$.
  \begin{enumerate}
  \item If $k=1$ and $d(x_0,x_1)=3$, then $F(x_0,x_1) = \iota(f(x_0))$.
  \item If $k\ge 2$ and $d(x_{k-1},x_k)=3$, then $F(x_0,\ldots,x_k)=\iota(x_{k-2})$.
  \item If $k=1$ and $d(x_0,x_1)=2$, then $F(x_0,x_1) = \iota(g_L(x_0,x_1))$.
  \item If $k\ge 2$, $d(x_{k-1},x_k)=2$, $d(x_{k-2},x_k)\ne 3$, then $F(x_0,\ldots,x_k)=\iota(\xi(x_{k-2},x_{k-1},x_k))$.
  \item If $k=2$, $d(x_1,x_2)=2$, $d(x_0,x_2)=3$, $x_1\ne f(x_0)$, then $F(x_0,x_1,x_2)=\iota(g_L(x_1,x_2))$.
  \item If $k\ge 3$, $d(x_{k-1},x_k)=2$, $d(x_{k-2},x_k)=3$, $x_{k-1} \ne x_{k-3}$, then $F(x_0,\ldots,x_k) = \iota(\ze(x_{k-3},x_{k-2},x_{k-1},x_k))$.
  \item If $k=2$, $d(x_0,x_2)=3$, and $x_1 = f(x_0)$, then $F(x_0,x_1,x_2) = \de$.
  \item If $k\ge 3$, $d(x_{k-2},x_k)=3$, and $x_{k-3}=x_{k-1}$, then $F(x_0,\ldots,x_k) = \de$.
  \item If $k=2$, $d(x_0,x_2)=2$, and $x_1 = g_L(x_0,x_2)$, then $F(x_0,x_1,x_2) = \de$.
  \item If $k\ge 3$, $d(x_{k-2},x_k)=2$, $d(x_{k-3},x_k)\ne 3$, and $x_{k-1} = \xi(x_{k-3},x_{k-2},x_k)$, then $F(x_0,\ldots,x_k) = \de$.
  \item If $k=3$, $d(x_1, x_3) = 2$, $d(x_0,x_3)=3$, $x_1 \ne f(x_0)$, and $x_2 = g_L(x_1,x_3)$, then $F(x_0,x_1,x_2,x_3)=\de$.
  \item If $k\ge 4$, $d(x_{k-2},x_k)=2$, $d(x_{k-3},x_k)=3$, $x_{k-4}\ne x_{k-2}$, and $x_{k-1} = \ze(x_{k-4},x_{k-3},x_{k-2},x_k)$, then $F(x_0,\ldots,x_k) = \de$.
  \end{enumerate}
  
  Let us prove that $F$ is a valid matching rule. Note that $F$ is a diagonal matching rule.\footnote{Strictly speaking, we perform induction on $\ell$ and prove at the same time that $F$ is valid and diagonal up to sequences with $\ell(x_0,\ldots,x_k)\le \ell$. For simplicity, we use the fact that $F$ is diagonal during the proof that $F$ is valid. In proof of validity of rule (6), we also use validity of rule (10). Because we are doing induction, this is not circular argument.}
  \begin{enumerate}
  \item Suppose $k=1$ and $d(x_0,x_1)=3$. Then $d(x_0,f(x_0))=1$, so $F(x_0,f(x_0))=\ep$ and $F(x_0,f(x_0),x_1)=\de$ by rule (7).
  \item Suppose $k\ge 2$ and $d(x_{k-1},x_k)=3$.
  Because $(x_0,\ldots,x_{k-1})$ is unmatched, we have $d(x_{k-2},x_{k-1})=1$.
  So $F(x_0,\ldots,x_{k-1},x_{k-2}) \ne \iota(*)$, and thus $F(x_0,\ldots,x_{k-1},x_{k-2}) = \ep$.
  Therefore $F(x_0,\ldots,x_{k-1},x_{k-2},x_k)=\de$ by rule (8).
  \item Suppose $k=1$ and $d(x_0,x_1)=2$. Then $d(x_0,g_L(x_0,x_1))=1$, so $F(x_0,g_L(x_0,x_1))=\ep$ and $F(x_0,g_L(x_0,x_1),x_1)=\de$ by rule (9).
  \item Suppose $k\ge 2$, $d(x_{k-1},x_k)=2$, and $d(x_{k-2},x_k)\ne 3$.
  Because $d(x_{k-1},\xi(x_{k-2},x_{k-1},x_k))=1$, we have $F(x_0,\ldots,x_{k-1},\xi(x_{k-2},x_{k-1},x_k))\ne \iota(*)$.
  Because $d(x_{k-2},\xi(x_{k-2},x_{k-1},x_k))\le 1$, we have $F(x_0,\ldots,x_{k-1},\xi(x_{k-2},x_{k-1},x_k))\ne \de$.
  So $F(x_0,\ldots,x_{k-1},\xi(x_{k-2},x_{k-1},x_k)) = \ep$ and therefore 
  $F(x_0,\ldots,x_{k-1},\xi(x_{k-2},x_{k-1},x_k), x_k) = \de$ by rule (10).
  \item Suppose $k=2$, $d(x_1,x_2)=2$, $d(x_0,x_2)=3$, and $x_1\ne f(x_0)$.
  Because $d(x_1,g_L(x_1,x_2))=1$, we have $F(x_0,x_1,g_L(x_1,x_2))\ne \iota(*)$.
  Because $x_1 = g_R(x_0,g_L(x_1,x_2))$, we have $F(x_0,x_1,g_L(x_1,x_2))\ne \de$.
  So $F(x_0,x_1,g_L(x_1,x_2))=\ep$ and thus $F(x_0,x_1,g_L(x_1,x_2),x_2) = \de$ by rule (11).
  \item Suppose $k\ge 3$, $d(x_{k-1},x_k)=2$, $d(x_{k-2},x_k)=3$, and $x_{k-1} \ne x_{k-3}$. 
  Because $(x_0,\ldots,x_{k-1})$ is unmatched, we have $d(x_{k-3},x_{k-2})=d(x_{k-2},x_{k-1})=1$.
  Write $u = \ze(x_{k-3},x_{k-2},x_{k-1},x_k)$. Because $d(x_{k-1}, u) = 1$, we have $F(x_0,\ldots,x_{k-1},u)\ne \iota(*)$.
  By analyzing all four cases of $x_{k-3}$, we can see that $d(x_{k-3},u)\ne 3$, and $\xi(x_{k-3},x_{k-2},u)\ne x_{k-1}$. So $F(x_0,\ldots,x_{k-2},u) = \iota(\xi(x_{k-3},x_{k-2},u)) \ne \iota(x_{k-1})$ by rule (4), and $F(x_0,\ldots,x_{k-1},u)\ne \de$.
  So $F(x_0,\ldots,x_{k-1},u) = \ep$, and thus $F(x_0,\ldots,x_{k-1},u,x_k) = \de$ by rule (12).
  \item Suppose $k=2$, $d(x_0,x_2)=3$, and $x_1 = f(x_0)$. Then $F(x_0,x_2) = \iota(x_1)$ by rule (1).
  \item Suppose $k\ge 3$, $d(x_{k-2},x_k)=3$, and $x_{k-3}=x_{k-1}$. Then $F(x_0,\ldots,x_{k-2},x_k) = \iota(x_{k-1})$ by rule (2).
  \item Suppose $k=2$, $d(x_0,x_2)=2$, and $x_1 = g_L(x_0,x_2)$. Then $F(x_0,x_2) = \iota(x_1)$ by rule (3).
  \item Suppose $k\ge 3$, $d(x_{k-2},x_k)=2$, $d(x_{k-3},x_k)\ne 3$, and $x_{k-1} = \xi(x_{k-3},x_{K-2},x_k)$. Then $F(x_0,\ldots,x_{k-2},x_k) = \iota(x_{k-1})$ by rule (4).
  \item Suppose $k=3$, $d(x_1, x_3) = 2$, $d(x_0,x_3)=3$, $x_1 \ne f(x_0)$, and $x_2 = g_L(x_1,x_3)$. Then $F(x_0,x_1,x_3) = \iota(x_2)$ by rule (5).
  \item Suppose $k\ge 4$, $d(x_{k-2},x_k)=2$, $d(x_{k-3},x_k)=3$, $x_{k-4}\ne x_{k-2}$, and $x_{k-1} = \ze(x_{k-4},x_{k-3},x_{k-2},x_k)$. Then $F(x_0,\ldots,x_{k-2},x_k) = \iota(x_{k-1})$ by rule (6).
  \end{enumerate}
  
  Let $M$ be the prefix matching generated by $F$. Let us prove that $M$ is a Morse matching.
  Work in the setting of Lemma \ref{LemmaNonMorse} and Lemma \ref{LemmaNonMorseDiag}. 
  Suppose in $\Ga_{\MC_{*,*}(G)}^M$ there is a directed cycle $a_1 \to b_1 \to \cdots \to b_p \to a_{p+1} = a_1$. Define $d_i,c_i,u_i$ accordingly.
  By rotating, we can assume WLOG that $c_1$ is the smallest $c_i$ among all $i$'s.
  By Corollary \ref{CoroNonMorseDiagDA}, we have $d_i = c_i+1$ for all $i$.
  Let us do a case analysis depending on which rule the edge $b_1 \to a_2$ uses.
  
  \textbf{Case (1).} The edge $b_1 \to a_2$ uses rule (1).
  Then $d(a_{2,d_1},a_{2,d_1+1}) = 2$.
  By Lemma \ref{LemmaNonMorse}, we have $d_2 = d_1+1$.
  By Corollary \ref{CoroNonMorseDiagDA}, we have $d(a_{2,d_2-1},a_{2,d_2})=1$.
  Contradiction.
  
  \textbf{Case (2).} The edge $b_1 \to a_2$ uses rule (2).
  Then $d(a_{2,d_1},a_{2,d_1+1}) = 2$.
  By Lemma \ref{LemmaNonMorse}, we have $d_2 = d_1+1$.
  By Corollary \ref{CoroNonMorseDiagDA}, we have $d(a_{2,d_2-1},a_{2,d_2})=1$.
  Contradiction.
  
  \textbf{Case (3).} The edge $b_1 \to a_2$ uses rule (3). There are two sub-cases.
  
  \textbf{Sub-case (3.1).} There does not exist $i$ such that $d_i=3$.
  Then by Lemma \ref{LemmaNonMorse} and Lemma \ref{LemmaNonMorseDiag}, and that $d_i\ge 1$ for all $i$, we have $d_1=1$, $d_2=2$, $d_3=1$, $d_4=2$, $\cdots$.
  By checking all possible relative positions of $(a_{1,0},a_{1,1},a_{1,2},a_{1,3})$, we see that a cycle is not formed.\footnote{There are only a few possible relative positions. Same for the other checks.}
  
  \textbf{Sub-case (3.2).} There exist $i$ such that $d_i = 3$. Take the smallest such $i$.
  By Lemma \ref{LemmaNonMorse} and Lemma \ref{LemmaNonMorseDiag}, there exists $i^\p$ with $d_{i^\p}=2$. Then by Corollary \ref{CoroNonMorseDiagDA}, for all $j$, we have $d(a_{j,0},a_{j,1}) = d(a_{j,1},a_{j,2})=d(a_{j,2},a_{j,3})=1$.
  
  By checking all possible relative positions of $(a_{i,0},\ldots,a_{i,4})$, we see that $d(a_{i+1,3},a_{i+1,0})\le 2$.
  So $d(a_{j,3},a_{j,0}) \le 2$ for all $j \ge i+1$.
  
  Take the smallest $i^\p \ge i+1$ such that $d_{i^\p}=2$.
  By checking all possible relative positions of $(a_{i^\p,0},\ldots,a_{i^\p,3})$, we see that $d(a_{i^\p+1,2},a_{i^\p+1,0})\le 1$.
  So $d(a_{j,2},a_{j,0})\le 1$ for all $j\ge i^\p+1$.
  
  Take the smallest $i^\pp \ge i^\p+1$ such that $d_{i^\pp}=1$.
  There exist such $i^\pp$ because $d_{1+np}=d_1=1$ for all $n$.
  However, because $d(a_{i^\pp,2},a_{i^\pp,0})=1$, such $i^\pp$ cannot exist.
  Contradiction.
  
  \textbf{Case (4).} The edge $b_1 \to a_2$ uses rule (4).
  There are two sub-cases.
  
  \textbf{Sub-case (4.1).} $d(a_{2,d_1},a_{2,d_1-2}) = 0$.
  Let $i$ be the smallest integer $\ge 2$ such that $d_i = d_1$.
  Note that $d(a_{2,d_1},a_{2,d_1+1}) = 1$.
  
  For $2\le j<i$, if $d_j\ne d_1+1$, then $d(a_{j+1,d_1},a_{j+1,d_1+1})=d(a_{j,d_1},a_{j,d_1+1})=1$; if $d_j=d_1+1$, then $d(a_{j,d_1},a_{j,d_1+1})=1$ by Corollary \ref{CoroNonMorseDiagDA}.
  So $d(a_{i,d_1},a_{i,d_1+1})=1$ by induction.
  It is also clear that $d(a_{i,d_1},a_{i,d_1-2})=0$.
  
  So $d(b_{i,d_1-1},b_{i,d_1}) = d(a_{i,d_1-1},a_{i,d_1+1})=2$.
  Also, $d(b_{i,d_1},b_{i,d_1-2}) = d(a_{i,d_1+1},a_{i,d_1-2}) = d(a_{i,d_1+1},a_{i,d_1}) = 1$.
  So rule (4) applies to $(b_{i,0},\ldots,b_{i,d_1})$, and $F(b_{i,0},\ldots,b_{i,d_1}) = \iota(b_{i,d_1-2})$.
  This means $a_i=a_{i+1}$, and edges $a_i\to b_i$ and $b_i\to a_{i+1}$ cannot both exist in $\Ga_{\MC_{*,*}(G)}^M$. Contradiction.
  
  \textbf{Sub-case (4.2).} $d(a_{2,d_1},a_{2,d_1-2}) = 1$.
  Let $i$ be the smallest integer $\ge 2$ such that $d_i = d_1$.
  By the same reason as sub-case (4.1), we have $d(a_{i,d_1},a_{i,d_1+1})=1$, and $d(a_{i,d_1},a_{i,d_1-2})=1$.
  So $d(b_{i,d_1-1},b_{i,d_1}) = d(a_{i,d_1-1},a_{i,d_1+1})=2$.
  Also, $d(b_{i,d_1},b_{i,d_1-2}) = d(a_{i,d_1+1},a_{i,d_1-2}) \le  d(a_{i,d_1+1},a_{i,d_1}) + d(a_{i,d_1},a_{i,d_1-2})= 2$.
  So rule (4) applies to $(b_{i,0},\ldots,b_{i,d_1})$.
  
  If $a_{i+1,d_1} = a_{i,d_1}$, then $a_{i+1} = a_i$, and edges $a_i\to b_i$ and $b_i\to a_{i+1}$ cannot both exist in $\Ga_{\MC_{*,*}(G)}^M$.
  So $a_{i+1,d_1}\ne a_{i,d_1}$. Then by rule (4) and properties of $\xi$, we have $d(a_{i+1,d_1},a_{i+1,d_1-2}) < d(a_{i,d_1},a_{i,d_1-2})=1$.
  We can rotate $a_i$ to $a_1$ and reduce to Sub-case (4.1).
  
  \textbf{Case (5).} The edge $b_1 \to a_2$ uses rule (5).
  Let $i$ be the smallest integer $\ge 2$ such that $d_i = d_1$.

  If $a_{i,d_i+1} = a_{2,d_i+1}$, then $(a_{2,0},\ldots,a_{2,d_i+1}) = (a_{i,0},\ldots,a_{i,d_i+1})$.
  So $(b_{1,0},\ldots,b_{1,d_i}) = (b_{i,0},\ldots,b_{i,d_i})$, and by the edge $b_1 \to a_2$, we have $(a_{i+1,0},\ldots,a_{i+1,d_i+1}) = (a_{2,0},\ldots,a_{2,d_i+1})$.
  Then $a_i = a_{i+1}$, and edges $a_i \to b_i$ and $b_i\to a_{i+1}$ cannot both exist in $\Ga_{\MC_{*,*}(G)}^M$. Contradiction.
  
  So $a_{i,d_i+1} \ne a_{2,d_i+1}$. This means $d(b_{i,d_i}, b_{i,d_i-2})\ne 3$. We can rotate $a_i$ to $a_1$ and reduce to Case (4).

  \textbf{Case (6).} The edge $b_1 \to a_2$ uses rule (6).
  Let $i$ be the smallest integer $\ge 2$ such that $d_i = d_1$.
  By the same reason as Case (5), we have $a_{i,d_i+1} \ne a_{2,d_i+1}$ and $d(b_{i,d_i}, b_{i,d_i-2})\ne 3$. We can rotate $a_i$ to $a_1$ and reduce to Case (4).
  
  So all cases lead to contradiction. Therefore $M$ is a Morse matching.
  
  By Corollary \ref{CoroDiagMatRuleImplyGraph}, the icosahedral graph is diagonal.
  \end{proof}
  \subsection{Odd cycles}
  Hepworth and Willerton \cite{HW17} computed the first magnitude homology groups of small cycles using computer program, and made conjectures for magnitude homology groups of cycles in general.
  We prove their conjectures using algebraic Morse theory.
  We compute for odd cycles in this section, and for even cycles in the next section.
  It turns out that magnitude homology of odd cycles has more complicated description but easier computation.
  \begin{thm}\label{ThmOddCycle}
  Fix an integer $m\ge 2$. 
  The magnitude homology of $C_{2m+1}$ is described as follows.
  \begin{enumerate}
  \item All groups $\MH_{k,l}(C_{2m+1})$ are torsion-free.
  \item Define a function $T: \bZ\t \bZ \to \bZ$ as
  \begin{enumerate}
  \item $T(k,l)=0$ if $k<0$ or $l<0$;
  \item $T(0,0) = 2m+1$, $T(1, 1) = 4m+2$;
  \item $T(k,l) = T(k-1,l-1) + 2T(k-2,l-(m+1))$ for $(k,l)\ne (0,0)$ and $(1,1)$.
  \end{enumerate}
  Then $\rk \MH_{k,l}(C_{2m+1}) = T(k,l)$ for all $k$ and $l$.
  \end{enumerate}
  \end{thm}
  \begin{proof}
  Label the vertices of $G=C_{2m+1}$ as $\{1,\ldots,2m+1\}$ such that $i$ is adjacent to $i+1$ for all $i$, and $2m+1$ adjacent to $1$.
  We define signed distance $\De : V(G) \t V(G) \to \{-m,\ldots,m\}$ such that $\De(u,v)$ is the unique number in $\{-m,\ldots,m\}$ in the same modulo-$(2m+1)$ equivalence class as $v-u$.
  Define $\sgn: \bZ \to \{-1,0,1\}$ that maps positive integers to $1$, negative integers to $-1$, and $0$ to $0$.
  It is easy to see that $|\De(u,v)| = d(u,v)$ and $\De(u,v) = -\De(v,u)$. We also have $d(u,v) + d(v,w) = d(u,w)$ if and only if $|\De(u,v)| + |\De(v,w)| \le m$ and $\sgn(\De(u,v)) = \sgn(\De(v,w))$.
  
  Define a function $\sm : \{(u,v)\in V(G) \t V(G) : d(u,v) \ge 2\} \to V(G)$ that maps $(u,v)$ to the unique vertex $w$ with $d(u,w)=1$ and $d(u,w)+d(w,v)=d(u,v)$.
  For three vertices $(u,v,w)$, let $\chi(u,v,w)$ denote the proposition $\sgn(\De(u,v)) = \sgn(\De(v,w))\land d(u,v)=1\land d(v,w)=m$.
  
  Let us describe the matching rule.
  Fix a sequence $(x_0,\ldots,x_k)$ with unmatched prefix $(x_0,\ldots,x_{k-1})$.
  \begin{enumerate}
  \item If $k\ge 2$ and $x_{k-1} = \sm(x_{k-2},x_k)$, then $F(x_0,\ldots,x_k) = \de$.
  \item If $k\ge 1$, $d(x_{k-1},x_k)\ge 2$, not $(k\ge 2\land x_{k-1}=\sm(x_{k-2},x_k))$, and not $(k\ge 2 \land \chi(x_{k-2},x_{k-1},x_k))$, then $F(x_0,\ldots,x_k)=\iota(\sm(x_{k-1},x_k))$.
  \end{enumerate}
  
  Let us prove that $F$ is a valid matching rule.
  \begin{enumerate}
  \item Suppose $k\ge 2$ and $x_{k-1} = \sm(x_{k-2},x_k)$. Clearly $d(x_{k-2},x_k) \ge 2$. If $k\ge 3$ and $x_{k-2} = \sm(x_{k-3},x_k)$, then $x_{k-2} = \sm(x_{k-3},x_{k-1})$, and $F(x_0,\ldots,x_{k-1}) = \de$, which is not true. So we have $\lnot (k\ge 3\land x_{k-2} = \sm(x_{k-3},x_k))$.
  If $k\ge 3$ and $\chi(x_{k-3},x_{k-2},x_k)$, then $x_{k-2} = \sm(x_{k-3},x_{k-1})$, and $F(x_0,\ldots,x_{k-1}) = \de$, which is not true.
  So $F(x_0,\ldots,x_{k-2},x_k) = \iota(\sm(x_{k-2},x_k)) = \iota(x_{k-1})$.
  \item Suppose $k\ge 1$, $d(x_{k-1},x_k)\ge 2$, not $(k\ge 2\land x_{k-1}=\sm(x_{k-2},x_k))$, and not $(k\ge 2 \land \chi(x_{k-2},x_{k-1},x_k))$.
  Because $d(x_{k-1},\sm(x_{k-1},x_k)) = 1$, we have $F(x_0,\ldots,x_{k-1},\sm(x_{k-1},x_k)) \ne \iota(*)$.
  If $k\ge 2$ and $x_{k-1} = \sm(x_{k-2},\sm(x_{k-1},x_k))$, then either $x_{k-1} = \sm(x_{k-2},x_k)$, or $(k\ge 2 \land \chi(x_{k-2},x_{k-1},x_k))$, neither of which is true.
  So we have $F(x_0,\ldots,x_{k-1},\sm(x_{k-1},x_k)) = \ep$ and $F(x_0,\ldots,x_{k-1},\sm(x_{k-1},x_k),x_k) = \de$.
  \end{enumerate}
  
  The unmatched sequences are described as follows.
  \begin{enumerate}
  \item $(v)$ is unmatched for any vertex $v$.
  \item $(u,v)$ is unmatched for $d(u,v)=1$.
  \item If $(x_0,\ldots,x_k)$ is unmatched with $k\ge 1$ and $d(x_{k-1},x_k)=1$, then 
  $(x_0,\ldots,x_k,v)$ is unmatched, where $d(x_k,v) = m$ and $\sgn(\De(x_{k-1},x_k)) = \sgn(\De(x_k,v))$.
  \item If $(x_0,\ldots,x_k)$ is unmatched with $k\ge 1$ and $d(x_{k-1},x_k)=m$, then 
  $(x_0,\ldots,x_k,v)$ is unmatched, where $d(x_k,v) = 1$ and $\sgn(\De(x_{k-1},x_k)) = \sgn(\De(x_k,v))$.
  \item If $(x_0,\ldots,x_k)$ is unmatched with $k\ge 1$, then 
  $(x_0,\ldots,x_k,v)$ is unmatched, where $d(x_k,v) = 1$ and $\sgn(\De(x_{k-1},x_k)) \ne \sgn(\De(x_k,v))$.
  \end{enumerate}
  Note that no unmatched sequences have outgoing edges in $\Ga_{\MC_{*,*}(G)}$.
  
  Let $M$ be the prefix matching generated by $F$.
  Let us prove that $M$ is a Morse matching.
  Work in the setting of Lemma \ref{LemmaNonMorse}.
  Suppose in $\Ga_{\MC_{*,*}(G)}^M$ there is a directed cycle $a_1 \to b_1 \to \cdots \to b_p \to a_{p+1} = a_1$. Define $d_i,c_i,u_i$ accordingly.
  Although $F$ is not a diagonal matching rule, it satisfies some good properties that a diagonal matching rule has. We prove the following analogue of Lemma \ref{LemmaNonMorseDiag}.
  \begin{lemma}\label{LemmaOddCycleSub}
  We have $d_i = c_i+1$ for all $i$.
  \end{lemma}
  \begin{proof}
  Suppose $(b_{i,0},\ldots,b_{i,d_i})$ is unmatched.
  Then $c_i\ge d_i$, $(a_{i+1,0},\ldots,a_{i+1,d_i}) = (b_{i,0},\ldots,b_{i,d_i})$, and $(a_{i+1,0},\ldots,a_{i+1,d_i+1})$ is unmatched (by valid property (1) in Definition \ref{DefnMatRule}).
  Because an unmatched sequence has no outgoing edges,
  we have $d_{i+1} \ge d_i+1$ and $(b_{i+1,0},\ldots,b_{i+1,d_i}) = (b_{i,0},\ldots,b_{i,d_i})$.
  
  Suppose for some $j\ge i+1$, we have $(b_{j,0},\ldots,b_{j,d_i}) = (b_{i,0},\ldots,b_{i,d_i})$.
  By the same reasoning, we have $d_{j+1} \ge d_i+1$, and $(b_{j+1,0},\ldots,b_{j+1,d_i}) = (b_{i,0},\ldots,b_{i,d_i})$.
  Applying induction, we see that $d_j\ge d_i+1$ for all $j\ge i+1$.
  This means $d_{i+p}\ge d_i+1$, which cannot be true.
  
  So $(b_{i,0},\ldots,b_{i,d_i})$ is matched.
  By definition of prefix matchings, the matching state of $(b_{i,0},\ldots,b_{i,d_i})$ is the same as that of $b_i$, which is insert($c_i$, $u_i$).
  In particular, $c_i+1 \le d_i$.
  By Lemma \ref{LemmaNonMorse}, we have $d_i \le c_i+1$.
  So $d_i = c_i+1$.
  \end{proof}
  Now we return to the proof that $M$ is a Morse matching.
  By Lemma \ref{LemmaOddCycleSub} and Lemma \ref{LemmaNonMorseDC}, we have $d(a_{1,d_1-1},a_{1,d_1})=1$.
  Because of the edge $a_1 \to b_1$, we have $d(a_{1,d_1-1}, a_{1,d_1}) + d(a_{1,d_1}, a_{1,d_1+1}) = d(a_{1,d_1-1}, a_{1,d_1+1})$.
  So $a_{1,d_1} = \sm(a_{1,d_1-1}, a_{1,d_1+1})$.
  Also, by the edge $b_1 \to a_2$, we have $a_{2,c_1+1} = \sm(a_{2,c_1}, a_{2,c_1+2})$.
  By Lemma \ref{LemmaOddCycleSub}, we have $c_1 = d_1-1$.
  Also, we have $a_{1,j} = b_{1,j} = a_{2,j}$ for $j\le d_1-1$, and $a_{1,j} = b_{1,j-1} = a_{2,j}$ for $j\ge d_1+1$.
  So $a_{2,d_1} = a_{1,d_1}$, and therefore $a_2 = a_1$.
  Then edges $a_1 \to b_1$ and $b_1\to a_2$ cannot both exist in $\Ga_{\MC_{*,*}(G)}^M$. Contradiction.
  So $M$ is a Morse matching.
  
  Because no unmatched sequences have any outgoing edges, all differentials $\der^\circ : \MC_{k,l}^\circ(G)\to \MC_{k-1,l}^\circ(G)$ are zeros.
  Therefore the homology classes of unmatched sequences form a basis of $\MH_{*,*}(G)$.
  It is not hard to see that $\rk \MH_{*,*}(G)$ is as described in the theorem statement.
  \end{proof}
  \subsection{Even cycles}
  \begin{thm}\label{ThmEvenCycle}
  Fix an integer $m\ge 3$.
  The magnitude homology of $C_{2m}$ is described as follows.
  \begin{enumerate}
  \item All groups $\MH_{k,l}(C_{2m})$ are torsion-free.
  \item Define a function $T: \bZ \t \bZ \to \bZ$ as
  \begin{enumerate}
  \item $T(k,l)=0$ if $k<0$ or $l<0$;
  \item $T(0,0) = 2m$, $T(1,1) = 4m$;
  \item $T(k,l) = \max\{T(k-1,l-1), T(k-2,l-m)\}$ for $(k,l)\ne (0,0)$ and $(1,1)$.
  \end{enumerate}
  Then $\rk \MH_{k,l}(C_{2m}) = T(k,l)$ for all $k$ and $l$.
  \end{enumerate}
  \end{thm}
  \begin{proof}
  Label the vertices of $G=C_{2m}$ as $\{1,\ldots,2m\}$ such that $i$ is adjacent to $i+1$ for all $i$, and $2m$ adjacent to $1$. We define a signed distance $\De: V(G) \t V(G) \to \{-m+1,\ldots,m\}$ such that $\De(u,v)$ is the unique number in $\{-m+1,\ldots,m\}$ in the same modulo-$(2m)$ equivalence class as $v-u$. Define $\sgn: \bZ \to \{-1,0,1\}$ that maps positive integers to $1$, negative integers to $-1$, and $0$ to $0$.
  
  Define a function $\sm: \{(u,v) \in V(G) \t V(G) : d(u,v) \ge 2\}$ that maps $(u,v)$ to the unique vertex $w$ with $d(u,w)=1$ and $\sgn(\De(u,w)) = \sgn(\De(u,v))$.
  Note that when $d(u,v)\le m-1$, $\sm(u,v)$ is the unique vertex $w$ with $d(u,w)=1$ and $d(u,w)+d(w,v)=d(u,v)$.
  For three vertices $(u,v,w)$, let $\chi(u,v,w)$ denote the proposition $\De(u,v)=-1 \land \De(v,w)=-m+1$.
  
  Let us describe the matching rule. Fix a sequence $(x_0,\ldots,x_k)$ with unmatched prefix $(x_0,\ldots,x_{k-1})$.
  \begin{enumerate}
  \item If $k\ge 2$ and $x_{k-1} = \sm(x_{k-2},x_k)$, then $F(x_0,\ldots,x_k) = \de$.
  \item If $k\ge 3$, $\De(x_{k-3},x_{k-2})=1$, and $\chi(x_{k-2},x_{k-1},x_k)$, then $F(x_0,\ldots,x_k) = \de$.
  \item If $k\ge 1$, $d(x_{k-1},x_k)\ge 2$, not $(k\ge 2 \land x_{k-1} = \sm(x_{k-2},x_k))$, not $(k\ge 2\land \chi(x_{k-2},x_{k-1},x_k))$, and not $(k\ge 2 \land \chi(x_{k-1},x_{k-2},x_k))$, then $F(x_0,\ldots,x_k) = \iota(\sm(x_{k-1},x_k))$.
  \item If $k\ge 2$ and $\chi(x_{k-1},x_{k-2},x_k)$, then $F(x_0,\ldots,x_k) = \iota(x_{k-2})$.
  \end{enumerate}
  
  Let us prove that $F$ is a valid matching rule.
  \begin{enumerate}
  \item Suppose $k\ge 2$ and $x_{k-1} = \sm(x_{k-2},x_k)$.
  Clearly $d(x_{k-2},x_k) \ge 2$.
  If $k\ge 3$ and $x_{k-2} = \sm(x_{k-3},x_k)$, then $x_{k-2} = \sm(x_{k-3},x_{k-1})$ and $F(x_0,\ldots,x_{k-1}) = \de$, which is not true. So we have $\lnot (x\ge 3 \land x_{k-2} = \sm(x_{k-3},x_k))$.
  If $k\ge 3$ and $\chi(x_{k-3},x_{k-2},x_k)$, then $x_{k-2} = \sm(x_{k-3},x_{k-1})$ and $F(x_0,\ldots,x_{k-1}) = \de$, which is not true.
  So we have $\lnot (k\ge 4 \land \chi(x_{k-3},x_{k-2},x_k))$.
  If $k\ge 3$ and $\chi(x_{k-2},x_{k-3},x_k)$, then $x_{k-2} = \sm(x_{k-3},x_{k-1})$ and $F(x_0,\ldots,x_{k-1}) = \de$, which is not true.
  So $F(x_0,\ldots,x_{k-2},x_k) = \iota(\sm(x_{k-2},x_k)) = \iota(x_{k-1})$ by rule (3).
  \item Suppose $k\ge 3$, $\De(x_{k-3},x_{k-2})=1$, and $\chi(x_{k-2},x_{k-1},x_k)$.
  Then $x_{k-3} = x_{k-1}$, and $F(x_0,\ldots,x_{k-2},x_k) = \iota(x_{k-3}) = \iota(x_{k-1})$ by rule (4).
  \item Suppose $k\ge 1$, $d(x_{k-1},x_k)\ge 2$, not $(k\ge 2 \land x_{k-1} = \sm(x_{k-2},x_k))$, not $(k\ge 2 \land \chi(x_{k-2},x_{k-1},x_k))$, and not $(k\ge 2 \land \chi(x_{k-1},x_{k-2},x_k))$.
  Because $d(x_{k-1},\sm(x_{k-1},x_k))=1$, we have $F(x_0,\ldots,x_{k-1},\sm(x_{k-1},x_k)) \ne \iota(*)$.
  If $k\ge 2$ and $x_{k-1} = \sm(x_{k-2}, \sm(x_{k-1},x_k))$, then there are three possibilities.
  \begin{enumerate}
  \item $x_{k-1} = \sm(x_{k-2},x_k)$;
  \item $\chi(x_{k-1},x_{k-2},x_k)$;
  \item $\chi(x_{k-2},x_{k-1},x_k)$. 
  \end{enumerate}
  None of these can be true.
  So we have $\lnot (k\ge 2 \land x_{k-1} = \sm(x_{k-2},\sm(x_{k-1},x_k)))$.
  Because $m\ge 3$, rule (2) does not apply to the sequence $(x_0,\ldots,x_{k-1},\sm(x_{k-1},x_k))$.
  So $F(x_0,\ldots,x_{k-1},\sm(x_{k-1},x_k))\ne \de$.
  Therefore $F(x_0,\ldots,x_{k-1},\sm(x_{k-1},x_k)) = \ep$ and thus $F(x_0,\ldots,x_{k-1},\sm(x_{k-1},x_k),x_k) = \de$ by rule (1).
  \item Suppose $k\ge 2$ and $\chi(x_{k-1},x_{k-2},x_k)$.
  Because $d(x_{k-2},x_{k-1})=1$, we have $F(x_0,\ldots,x_{k-1},x_{k-2}) \ne \iota(*)$.
  Clearly $F(x_0,\ldots,x_{k-1},x_{k-2})\ne \de$.
  So $F(x_0,\ldots,x_{k-1},x_{k-2}) = \ep$.
  Because $\chi(x_{k-1},x_{k-2},x_k)$ is true, we have $F(x_0,\ldots,x_{k-1},x_{k-2},x_k) = \de$ by rule (2).
  \end{enumerate}
  
  We say a sequence $(x_0,\ldots,x_k)$ is special if $k$ is even (can be zero), and $\chi(x_{2i},x_{2i+1},x_{2i+2})$ holds for all $i$.
  The unmatched sequences are described as follows.
  \begin{enumerate}
  \item A special sequence is unmatched.
  \item If $(x_0,\ldots,x_k)$ is a special sequence, then $(x_0,\ldots,x_k,v)$ is unmatched, where $d(x_k,v)=1$.
  \item If $(x_0,\ldots,x_k)$ is unmatched with $d(x_{k-1},x_k)=1$, then $(x_0,\ldots,x_k,v)$ is unmatched, where $d(x_k,v)=1$ and $\sgn(\De(x_{k-1},x_k))\ne \sgn(\De(x_k,v))$.
  \end{enumerate}
  
  For a sequence $(x_0,\ldots,x_k)$ define $\rho(x_0,\ldots,x_k)$ to be the largest integer $j$ such that $(x_0,\ldots,x_j)$ is a special sequence.
  
  Let $M$ be the prefix matching generated by $F$.
  Let us prove that $M$ is a Morse matching.
  Work in the setting of Lemma \ref{LemmaNonMorse}.
  Suppose in $\Ga_{\MC_{*,*}(G)}^M$ there is a directed cycle $a_1\to b_1 \to \cdots \to b_p \to a_{p+1} = a_1$. Define $d_i$, $c_i$, $u_i$ accordingly.
  \begin{lemma}\label{LemmaEvenCycleSub1}
  We have $d_i \ge \rho(a_i)+1$ for all $i$.
  \end{lemma}
  \begin{proof}
  Because special sequences are unmatched, we have $c_i \ge \rho(b_i)$ and $\rho(a_{i+1}) = \rho(b_i)$ for all $i$.
  If $d_i\le \rho(a_i)$ for some $i$, then $\rho(a_{i+1}) = \rho(b_i) = 2\lfloor \frac {d_i-1}2\rfloor$.
  If $d_i\ge \rho(a_i)+1$, then $\rho(a_{i+1}) = \rho(b_i) \ge \rho(a_i)$.
  
  Suppose $d_1 \le \rho(a_1)$, and that $d_1$ is the smallest among all $d_i\le \rho(a_i)$.
  Then $\rho(a_2) = 2\lfloor \frac{d_1-1}2\rfloor$ is the smallest among all $\rho(a_i)$'s.
  So $d_i\ge 2\lfloor \frac{d_1-1}2\rfloor+1$ for all $i$.
  There are two cases depending on parity of $d_1$.
  
  \textbf{Case 1: $d_1$ is odd.}
  Because $(a_{1,0},\ldots,a_{1,d_1+1})$ is a special sequence, it is easy to check that the conditions of rule (3) holds for $(b_{1,0},\ldots,b_{1,d_1})$.
  So $c_1 = d_1-1$, $F(b_{1,0},\ldots,b_{1,d_1}) = \iota(\sm(b_{1,d_1-1},b_{1,d_1}))$, and $\De(a_{2,d_1-1},a_{2,d_1}) = 1$.
  
  Let us prove by induction that $(a_{i,0},\ldots,a_{i,d_1}) = (a_{2,0},\ldots,a_{2,d_1})$ for all $i\ge 2$.
  The case $i=2$ is trivial.
  Suppose $(a_{i,0},\ldots,a_{i,d_1}) = (a_{2,0},\ldots,a_{2,d_1})$.
  Because $d_1$ is odd, we have $d_i \ge d_1$.
  
  Suppose $d_i = d_1$.
  Because $\De(a_{i,d_1-1},a_{i,d_1})=1$ and $d(a_{i,d_1-1},a_{i,d_1}) + d(a_{i,d_1},a_{i,d_1+1}) = d(a_{i,d_1-1}, a_{i,d_1+1})$, we have $a_{i,d_1} = \sm(a_{i,d_1-1}, a_{i,d_1+1})$.
  Because $(a_{i,0},\ldots,a_{i,d_1-1}) = (a_{1,0},\ldots,a_{1,d_1-1})$ is a special sequence, we can verify that the conditions of rule (3) hold for $(b_{i,0},\ldots,b_{i,d_1})$.
  So $c_i = d_i-1$ and $F(b_{i,0},\ldots,b_{i,d_1}) = \iota(\sm(b_{1,d_1-1},b_{1,d_1}))$.
  This implies $a_{i+1} = a_i$, and edges $a_i\to b_i$ and $b_i\to a_{i+1}$ cannot both exist in $\Ga_{\MC_{*,*}(G)}^M$. Contradiction.
  
  So $d_i > d_1$.
  By Lemma \ref{LemmaNonMorse}, we have $c_i \ge d_i-1 \ge d_1$ and $(a_{i+1,0},\ldots,a_{i+1,d_1}) = (a_{i,0},\ldots,a_{i,d_1})$. This completes the induction step.
  
  So $(a_{i,0},\ldots,a_{i,d_1}) = (a_{2,0},\ldots,a_{2,d_1})$ for all $i\ge 2$.
  However, this cannot be true for $i=p+1$ because $a_{1,d_1}\ne a_{2,d_1}$. Contradiction.
  
  \textbf{Case 2: $d_1$ is even.}
  Because $\De(a_{1,d_1-1},a_{1,d_1}) = -m+1$ and $d(a_{1,d_1-1},a_{1,d_1}) + d(a_{1,d_1},a_{1,d_1+1}) = d(a_{1,d_1-1}, a_{1,d_1+1})$, we know that $d(a_{1,d_1-1},a_{1,d_1+1})=m$.
  Because $(a_{1,0},\ldots,a_{1,d_1})$ is a special sequence, it is easy to check that the conditions of rule (3) holds for $(b_{1,0},\ldots,b_{1,d_1})$.
  So $c_1 = d_1-1$, $F(b_{1,0},\ldots,b_{1,d_1}) = \iota(\sm(b_{1,d_1-1}, b_{1,d_1}))$, and $\De(a_{2,d_1-1}, a_{2,d_1}) = 1$.
  (Note that $a_{2,d_1} = a_{2,d_1-2}$.)
  
  Let us prove by induction that $(a_{i,0},\ldots,a_{i,d_1}) = (a_{2,0},\ldots,a_{2,d_1})$ for all $i\ge 2$.
  The case $i=2$ is trivial.
  Suppose $(a_{i,0},\ldots,a_{i,d_1}) = (a_{2,0},\ldots,a_{2,d_1})$.
  Because $d_1$ is odd, we have $d_i \ge d_1-1$.
  Because $a_{i,d_1} = a_{i,d_1-2}$, we have $d_i \ne d_1-1$.
  By the same reason as Case 1, we have $d_i \ne d_1$.
  So $d_i > d_1$. By Lemma \ref{LemmaNonMorse}, we have $c_i \ge d_i-1 \ge d_1$ and $(a_{i+1,0},\ldots,a_{i+1,d_1}) = (a_{i,0},\ldots,a_{i,d_1})$. This completes the induction step.
  
  So $(a_{i,0},\ldots,a_{i,d_1}) = (a_{2,0},\ldots,a_{2,d_1})$ for all $i\ge 2$.
  However, this cannot be true for $i=p+1$ because $a_{1,d_1}\ne a_{2,d_1}$. Contradiction.
  
  So both Case 1 and Case 2 lead to contradiction.
  \end{proof}
  \begin{lemma}\label{LemmaEvenCycleSub2}
  We have $d_i = c_i+1$ for all $i$.
  \end{lemma}
  \begin{proof}
  The proof mimics that of Lemma \ref{LemmaOddCycleSub}.
  Suppose $(b_{i,0},\ldots,b_{i,d_i})$ is unmatched.
  Then $c_i\ge d_i$, $(a_{i+1,0},\ldots,a_{i+1,d_i}) = (b_{i,0},\ldots,b_{i,d_i})$, and $(a_{i+1,0},\ldots,a_{i+1,d_i+1})$ is unmatched (by valid property (1) in Definition \ref{DefnMatRule}).
  By Lemma \ref{LemmaEvenCycleSub1}, we have $d_{i+1}\ge \rho(a_{i+1})+1$.
  By analyzing unmatched sequences we can see that $d_{i+1}\ge d_i+1$, and $(b_{i+1,0},\ldots,b_{i+1,d_i}) = (b_{i,0},\ldots,b_{i,d_i})$.
  
  Suppose for some $j\ge i+1$, we have $(b_{j,0},\ldots,b_{j,d_i}) = (b_{i,0},\ldots,b_{i,d_i})$. By the same reasoning, we have $d_{j+1} \ge d_i+1$, and $(b_{j+1,0},\ldots,b_{j+1,d_i}) = (b_{i,0},\ldots,b_{i,d_i})$. Applying induction, we see that $d_j\ge d_i+1$ for all $j\ge i+1$.
  This means $d_{i+p}\ge d_i+1$, which cannot be true.
  
  So $(b_{i,0},\ldots,b_{i,d_i})$ is matched. 
  By definition of prefix matchings, the matching state of $(b_{i,0},\ldots,b_{i,d_i})$ is the same as that of $b_i$, which is insert($c_i$, $u_i$). In particular, $c_i+1\le d_i$.
  By Lemma \ref{LemmaNonMorse}, we have $d_i\le c_i+1$. So $d_i = c_i+1$.
  \end{proof}
  Now we return to the proof that $M$ is a Morse matching.
  By rotating, we can WLOG assume that $c_1$ is the smallest $c_i$ among all $i$'s.
  By Lemma \ref{LemmaEvenCycleSub2} and Lemma \ref{LemmaNonMorseDC}, we have $d(a_{i,d_i-1},a_{i,d_i})=1$ for all $i$.
  
  By Lemma \ref{LemmaNonMorse}, we have $d_2 \le c_1+2$ and $d_2 \ne c_1+1$.
  So $d_2 = c_1+2=d_1+1$ and therefore $c_2 = d_1$ by Lemma \ref{LemmaEvenCycleSub2}.
  By definition of $F$, we have $d(a_{2,d_1-1},a_{2,d_1})=d(a_{3,d_2-1},a_{3,d_2})=1$.
  Because $d_2 = d_1+1$, we have $d(a_{3,d_1-1},a_{3,d_1})=1$.
  
  Let us prove by induction that $(a_{i,0},\ldots,a_{i,d_1}) =(a_{3,0},\ldots,a_{3,d_1})$ and $d(a_{i,d_1}, a_{i,d_1+1})=1$ for all $i\ge 3$.
  The case $i=3$ is trivial. Suppose $(a_{i,0},\ldots,a_{i,d_1}) = (a_{3,0},\ldots,a_{3,d_1})$ and $d(a_{i,d_1}, a_{i,d_1+1})=1$. By assumption, $d_i \ge d_1$.
  
  Suppose $d_i = d_1$.
  Because $d(a_{i,d_1-1},a_{i,d_1}) = d(a_{i,d_1},a_{i,d_1+1})=1$ and $d(a_{i,d_1-1},a_{i,d_1}) + d(a_{i,d_1},a_{i,d_1+1}) = d(a_{i,d_1-1},a_{i,d_1+1})$, we have $a_{i,d_1} = \sm(a_{i,d_1-1},a_{i,d_1+1})$.
  By Lemma \ref{LemmaEvenCycleSub2}, $c_i = d_i-1$.
  Because $d(a_{i,d_1-1},a_{i,d_1+1}) = 2 < m$, rule (4) does not apply to $(b_{i,0},\ldots,b_{i,d_1})$.
  So rule (3) must apply to $(b_{i,0},\ldots,b_{i,d_1})$.
  This means $F(b_{i,0},\ldots,b_{i,d_1}) = \iota(\sm(b_{i,d_1-1},b_{i,d_1})) = \iota(a_{i,d_1})$.
  This implies $a_{i+1}=a_i$, and edges $a_i\to b_i$ and $b_i\to a_{i+1}$ cannot both exist in $\Ga_{\MC_{*,*}(G)}^M$. Contradiction.
  
  So $d_i > d_1$. By Lemma \ref{LemmaNonMorse}, we have $c_i \ge d_i-1\ge d_1$ and $(a_{i+1,0},\ldots,a_{i+1,d_i}) = (a_{3,0},\ldots,a_{3,d_i})$.
  If $d_i = d_1+1$, then by Lemma \ref{LemmaEvenCycleSub2}, $c_i = d_1$, and we have $d(a_{i+1,d_1},a_{i+1,d_1+1})=1$ by definition of $F$.
  If $d_i \ge d_1+2$, then $c_i\ge d_i-1\ge d_1+1$ and $d(a_{i+1,d_1},a_{i+1,d_1+1})=d(a_{i,d_1},a_{i,d_1+1})=1$.
  So in either case the induction step is completed.
  
  So $(a_{i,0},\ldots,a_{i,d_1}) =(a_{3,0},\ldots,a_{3,d_1})$ for all $i\ge 3$.
  However, this cannot be true for $i=p+1$ because $a_{1,d_1}\ne a_{2,d_1} = a_{3,d_1}$. Contradiction.
  So $M$ is a Morse matching.
  
  Let us analyze the differentials.
  Note that all sequences $(x_0,\ldots,x_k)$ in $I_{k,l}^\circ$ have $\rho(x_0,\ldots,x_k) = \frac{2(l-k)}{m-2}$.
  Because $\rho(x_0,\ldots,x_k)$ must be an even number, we have $I_{k,l}^\circ = \es$ if $(m-2) \nmid (l-k)$. 
  So when $m\ne 3$, there do not exist $k$ and $l$ such that $I_{k,l}^\circ$ and $I_{k-1,l}^\circ$ are both nonempty.
  
  Now suppose $m=3$. For all sequences $(x_0,\ldots,x_k)$ in $I_{k,l}^\circ$, we have
  $$\De(x_0,x_k) \in \{3(l-k) + w : w\in \{-1,0,1\}\} \pmod 6.$$
  Let $A_{k,l}$ denote the set $\{3(l-k) + w : w\in \{-1,0,1\}\} \bmod 6$.
  Then $A_{k,l}$ and $A_{k-1,l}$ are disjoint.
  On the other hand, if there is a zig-zag path $\ga \in \Ga_{x,y}^M$ for two sequences $(x_0,\ldots,x_k)$ and $(y_0,\ldots,y_{k-1})$, then we must have $x_0=y_0$ and $x_k=y_{k-1}$, and therefore $\De(x_0,x_k) = \De(y_0,y_{k-1})$.
  So $\Ga_{x,y}^M = \es$ for $x\in I_{k,l}^\circ$ and $y\in I_{k-1,l}^\circ$.
  
  So all differentials $\der^\circ : \MC_{k,l}^\circ(G) \to \MC_{k-1,l}^\circ(G)$ are zeros.
  Therefore the homology classes of unmatched sequences form a basis of $\MH_{*,*}(G)$. It is not hard to see that $\rk \MH_{*,*}(G)$ is as described in the theorem statement.
  \end{proof}
  
  \appendix
  \section{Magnitude homology is stronger than magnitude} \label{SecMHvsMag}
  Hepworth and Willerton \cite{HW17} asked whether there exist graphs with the same magnitude but different magnitude homology. In this appendix we answer the question in the affirmative by giving explicit examples.
  For computing magnitude homology, we use Sage and Python program \texttt{rational\_graph\_homology\_arxiv.py} written by Simon Willerton and James Cranch, which can be found in the arXiv version of \cite{HW17}.
  
  We follow notations of Leinster \cite{Lei17}.
  Let $G$ be a finite simple undirected connected graph. 
  Its magnitude $\#G$ is an element of $\bZ[[q]]\cap \bQ(q)$, i.e., it is both a power series with coefficients in $\bZ$, and a rational function.
  
  The following lemma is useful for proving two vertex-transitive graphs have the same magnitude.
  \begin{lemma}[Speyer, in Leinster \cite{Lei17}]\label{LemmaMagVT}
  Let $G$ be a vertex-transitive graph and $a \in V(G)$ be a fixed vertex.
  The magnitude of $G$ is given by
  $$\#G = \frac{|V(G)|}{\sum_{x\in V(G)} q^{d(a,x)}}.$$
  \end{lemma}
  
  The first example we give is the $4\t 4$ rook graph and the Shrikhande graph.
  The $4\t 4$ rook graph is the Cayley graph on $\bZ/4\bZ \t \bZ/4\bZ$ with generators $\{(0,x), (x,0) : x=1,2,3\}$. 
  The Shrikhande graph is the Cayley graph on $\bZ/4\bZ \t \bZ/4\bZ$ with generators $\{\pm(0,1), \pm(1, 0), \pm(1,1)\}$.
  
  \begin{figure}[h]
  \includegraphics[scale=0.3]{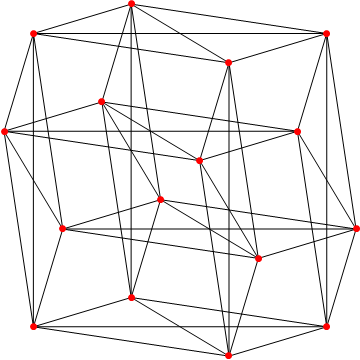}
  \hspace{1cm}
  \includegraphics[scale=0.3]{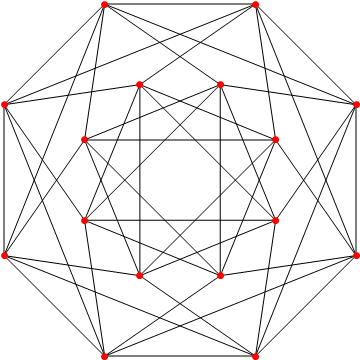}
  \caption{Left: $4\t 4$ rook graph; Right: Shrikhande graph}
  \end{figure}
  
  \begin{prop}
  The $4\t 4$ rook graph and the Shrikhande graph have the same magnitude but different magnitude homology.
  \end{prop}
  \begin{proof}
  Let $R_{4,4}$ denote the $4\t 4$ rook graph and let $\Shr$ denote the Shrikhande graph.
  Because $R_{4,4}$ and $\Shr$ are both Cayley graphs, they are both vertex transitive. So Lemma \ref{LemmaMagVT} applies.
  Simple calculation shows that for a fixed vertex in $R_{4,4}$, there are $6$ vertices with distance $1$ and $9$ vertices with distance $2$. The same holds for $\Shr$.
  Therefore $$\# R_{4,4} = \# \Shr = \frac{16}{1 + 6q + 9q^2}.$$
  
  On the other hand, computer computation shows that $\rk \MH_{3,4}(R_{4,4})=0$ while $\rk \MH_{3,4}(\Shr)\ne 0$ (Table \ref{TabMHR44} and Table \ref{TabMHShr}). Therefore $R_{4,4}$ and $\Shr$ have different magnitude homology.
\begin{table}[h]
\begin{tabular}{r|rrrrrrr}
&0&1&2&3&4&5&6\\ \hline
0&16\\
1&&96\\
2&&&432\\
3&&&&1728\\
4&&&&&6480\\
5&&&&&&23328\\
6&&&&&&&81648
\end{tabular}
\caption{Magnitude homology of the $4\t 4$ rook graph}
\label{TabMHR44}
\end{table}
\begin{table}[h]
\begin{tabular}{r|rrrrrrr}
&0&1&2&3&4&5&6\\ \hline
0&16\\
1&&96\\
2&&&432\\
3&&&&1728\\
4&&&&144&6624\\
5&&&&&1632&24960\\
6&&&&&&11824&93472
\end{tabular}
\caption{Magnitude homology of the Shrikhande graph}
\label{TabMHShr}
\end{table}
  \end{proof}
  
  Another example is the dodecahedral graph and the Desargues graph.
  \begin{figure}[h]
  \includegraphics[scale=0.3]{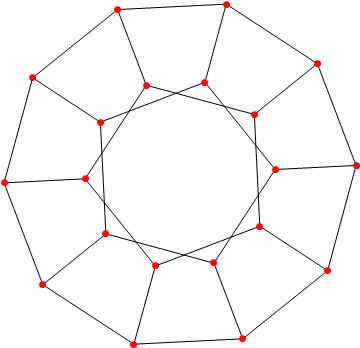}
  \hspace{1cm}
  \includegraphics[scale=0.3]{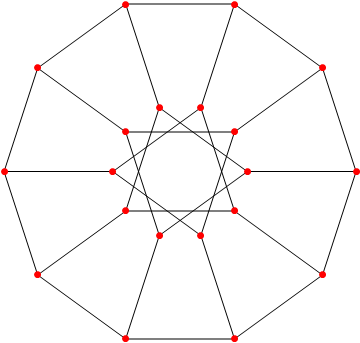}
  \caption{Left: dodecahedral graph; Right: Desargues graph}
  \end{figure}
  \begin{prop}
  The dodecahedral graph and the Desargues graph have the same magnitude but different magnitude homology.
  \end{prop}
  \begin{proof}
  Let $\Dod$ denote the dodecahedral graph and let $\Des$ denote the Desargues graph.
  It is not hard to see that both $\Dod$ and $\Des$ are vertex-transitive graphs, and therefore Lemma \ref{LemmaMagVT} applies.
  For a fixed vertex in $\Dod$, there are $3$ vertices with distance $1$, $6$ vertices with distance $2$, $6$ vertices with distance $3$, $3$ vertices with distance $4$ and $1$ vertex with distance $5$. The same holds for $\Des$.
  So we have
  $$\#\Dod=\#\Des=\frac{20}{1+3q+6q^2+6q^3+3q^4+q^5}.$$
  
  On the other hand, computer computation shows $\rk \MH_{2,4}(\Dod)\ne 0$ while $\rk \MH_{2,4}(\Des)=0$ (Table \ref{TabMHDod} and Table \ref{TabMHDes}). Therefore $\Dod$ and $\Des$ have different magnitude homology.
  \begin{table}[h]
\begin{tabular}{r|rrrrrrrrr}
&0&1&2&3&4&5&6&7&8\\ \hline
0&20\\
1&&60\\
2&&&60\\
3&&&120&60\\
4&&&60&360&60\\
5&&&&380&600&60\\
6&&&&60&1320&840&60\\
7&&&&&1020&3240&1080&60\\
8&&&&&180&4620&6120&1320&60
\end{tabular}
\caption{Magnitude homology of the dodecahedral graph}
\label{TabMHDod}
\end{table}
\begin{table}[h]
\begin{tabular}{r|rrrrrrrrr}
&0&1&2&3&4&5&6&7&8\\ \hline
0&20\\
1&&60\\
2&&&60\\
3&&&120&60\\
4&&&&300&60\\
5&&&&20&240&60\\
6&&&&&660&240&60\\
7&&&&&&1380&240&60\\
8&&&&&&300&900&240&60
\end{tabular}
\caption{Magnitude homology of the Desargues graph}
\label{TabMHDes}
\end{table}
  \end{proof}
  \section{Geodetic ptolemaic graphs} \label{SecGeoPto}
  In this appendix, we prove that graphs that are both ptolemaic and geodetic are diagonal using algebraic Morse theory. Recall that a graph is ptolemaic if for every four vertices $(x,y,z,w)$, we have Ptolemy's inequality $$d(x,y)d(z,w)+d(y,z)d(x,w) \ge d(x,z)d(y,w),$$ and a graph is geodetic if there is only one shortest path between any two vertices. Here we use an equivalent characterization of ptolemaic graphs.
  \begin{prop}
  Let $G$ be a graph. The following are equivalent.
  \begin{enumerate}
  \item $G$ is ptolemaic.
  \item For every four vertices $(x,y,z,w)$ with $y\ne z$, $d(x,y)+d(y,z)=d(x,z)$ and $d(y,z)+d(z,w)=d(y,w)$, we have $d(x,y)+d(y,z)+d(z,w) = d(x,w)$.
  \item For every four vertices $(x,y,z,w)$ with $d(x,y)+d(y,z)=d(x,z)$, $d(y,z)+d(z,w)=d(y,w)$, and $d(x,y)=d(y,z)=1$, we have $d(x,y)+d(y,z)+d(z,w) = d(x,w)$.
  \item $G$ is chordal and distance-hereditary. (Recall that a graph is chordal if it does not contain an induced cycle of length at least $4$, and a graph is distance-hereditary if every induced path is a shortest path.)
  \end{enumerate}
  \end{prop}
  \begin{proof}
  (1) $\imp$ (2): Suppose $G$ is ptolemaic and we have four vertices $(x,y,z,w)$ with $y\ne z$, $d(x,y)+d(y,z)=d(x,z)$ and $d(y,z)+d(z,w)=d(y,w)$.
  Expanding and simplifying Ptolemy's inequality, we get $d(x,w) \ge d(x,y)+d(y,z)+d(z,w)$. (This step uses $d(y,z)\ne 0$.)
  By triangle inequality, we have $d(x,w) \le d(x,y)+d(y,z)+d(z,w)$.
  So $d(x,w) = d(x,y)+d(y,z)+d(z,w)$.
  
  (2) $\imp$ (3): Obvious.
  
  (3) $\imp$ (4): Let us first prove a lemma.
  \begin{lemma} \label{LemmaPtoSub}
  Let $G$ be a graph satisfying (3).
  Let $k\ge 1$ and $(x_0,\ldots,x_k)$ be a sequence of vertices with $d(x_i,x_{i+1})=1$ and $d(x_i,x_{i+2})=2$ for all $i$. Then $d(x_0,x_k)=k$.
  \end{lemma}
  \begin{proof}
  Let us prove by induction.
  
  Induction base $k=1$ and $k=2$ are trivial.
  
  Assume $k\ge 3$ and that the result holds for $k^\p=2,3,\ldots,k-1$. Then we know $d(x_1,x_k) = k-1$ and $d(x_2,x_k) = k-2$ by induction hypothesis.
  Applying (3) to $(x_0,x_1,x_2,x_k)$, we get $d(x_0,x_k) = k$. This completes the induction step.
  \end{proof}
  
  Now we return to the proof of (3) $\imp$ (4).
  Suppose $G$ is not chordal and there is an induced cycle $(x_0,\ldots,x_k,x_{k+1}=x_0)$ with $k\ge 3$.
  Because this an induced cycle, we have $d(x_i,x_{i+2})=2$ for all $i$.
  By Lemma \ref{LemmaPtoSub}, $d(x_0,x_{k+1}) = k+1$. Contradiction. So $G$ is chordal. 
  That $G$ is distance-hereditary is immediate from Lemma \ref{LemmaPtoSub}.
  
  (4) $\iff$ (1): Proved by Howorka \cite{How81}.
  \end{proof}
  \begin{rmk}
  (3) is the characterization we use. (2) is a notion used in Leinster and Shulman \cite{LS17}. That (2) implies chordal is essentially known in op.~cit., Example 7.17.
  \end{rmk}
  
  \begin{thm}\label{ThmGeoPto}
  A geodetic ptolemaic graph is diagonal.
  \end{thm}
  \begin{proof}
  Let $G$ be a geodetic ptolemaic graph.
We define a function $\sm: \{(u,v) \in V(G) \t V(G) : d(u,v)\ge 2\} \to V(G)$ that maps $(u,v)$ to the unique vertex $w$ with $d(u,w)=1$ and $d(u,v) = d(u,w)+d(v,w)$.
  Existence and uniqueness follows from that $G$ is geodetic.
  
  Let us describe the matching rule.
  Fix a sequence $(x_0,\ldots,x_k)$ with unmatched prefix $(x_0,\ldots,x_{k-1})$.
  \begin{enumerate}
  \item If $k\ge 2$ and $x_{k-1} = \sm(x_{k-2},x_k)$, then $F(x_0,\ldots,x_k) = \de$.
  \item If $k\ge 1$, $d(x_{k-1}, x_k)\ge 2$, and not $(k\ge 2 \land x_{k-1} = \sm(x_{k-2},x_k))$, then $F(x_0,\ldots,x_k) = \iota(\sm(x_{k-1},x_k))$.
  \end{enumerate}
  
  Let us prove that $F$ is a valid matching rule. 
  Fix a sequence $(x_0,\ldots,x_k)$ with unmatched prefix $(x_0,\ldots,x_{k-1})$.
  \begin{enumerate}
  \item Suppose $k\ge 2$ and $x_{k-1} = \sm(x_{k-2},x_k)$. Clearly $d(x_{k-2},x_k)\ge 2$. Also, if $k\ge 3$ and $x_{k-2} = \sm(x_{k-3},x_k)$, then $x_{k-2} = \sm(x_{k-3}, x_{k-1})$ and $F(x_0,\ldots,x_{k-1}) = \de$, which is not true. So we have $\lnot (k\ge 3 \land x_{k-2} = \sm(x_{k-3},x_k))$.
  So $F(x_0,\ldots,x_{k-2},x_k) = \iota(\sm(x_{k-2},x_k)) = \iota(x_{k-1})$.
  \item Suppose $k\ge 1$, $d(x_{k-1}, x_k)\ge 2$, and not $(k\ge 2 \land x_{k-1} = \sm(x_{k-2},x_k))$.
  Because $d(x_{k-1},\sm(x_{k-1},x_k))=1$, we have $F(x_0,\ldots,x_{k-1},\sm(x_{k-1},x_k))\ne \iota(*)$.
  If $F(x_0,\ldots,x_{k-1},\sm(x_{k-1},x_k)) = \de$, then $k\ge 2$ and $x_{k-1} = \sm(x_{k-2}, \sm(x_{k-1},x_k))$.
  Applying ptolemaic characterization (3) to $(x_{k-2},x_{k-1},\sm(x_{k-1},x_k),x_k)$, we get $x_{k-1} = \sm(x_{k-2}, x_k)$, which is not true.
  So $F(x_0,\ldots,x_{k-1},\sm(x_{k-1},x_k)) = \ep$ and thus $F(x_0,\ldots,x_{k-1},\sm(x_{k-1},x_k), x_k) = \de$.
  \end{enumerate}
  
  It is easy to see that $F$ is a diagonal matching rule.
  
  Let $M$ be the prefix matching generated by $F$. Let us prove that $M$ is a Morse matching. Work in the setting of Lemma \ref{LemmaNonMorse}. Suppose in $\Ga_{\MC_{*,*}(G)}^M$ there is a directed cycle $a_1 \to b_1 \to \cdots \to b_p \to a_{p+1}=a_1$. 
  Define $d_i$, $c_i$, $u_i$ accordingly.
  
  By Corollary \ref{CoroNonMorseDiagDA}, we have $d(a_{1,d_1-1}, a_{1,d_1}) = 1$.
  Because of the edge $a_1 \to b_1$, we have $d(a_{1,d_1-1}, a_{1,d_1}) + d(a_{1,d_1}, a_{1,d_1+1}) = d(a_{1,d_1-1}, a_{1,d_1+1})$.
  So $a_{1,d_1} = \sm(a_{1,d_1-1}, a_{1,d_1+1})$.
  Also, by the edge $b_1 \to a_2$, we have $a_{2,c_1+1} = \sm(a_{2,c_1}, a_{2,c_1+2})$.
  By Lemma \ref{LemmaNonMorseDiag}, we have $c_1 = d_1-1$.
  Also, we have $a_{1,j} = b_{1,j} = a_{2,j}$ for $j\le d_1-1$, and $a_{1,j} = b_{1,j-1} = a_{2,j}$ for $j\ge d_1+1$.
  So $a_{2,d_1} = a_{1,d_1}$, and therefore $a_2 = a_1$.
  Then edges $a_1 \to b_1$ and $b_1\to a_2$ cannot both exist in $\Ga_{\MC_{*,*}(G)}^M$. Contradiction.
  So $M$ is a Morse matching.
  
  By Corollary \ref{CoroDiagMatRuleImplyGraph}, the graph $G$ is diagonal.
  \end{proof}
  
  It may sound like that Theorem \ref{ThmGeoPto} gives new diagonal graphs. However, by Kay and Chartrand \cite{KC65}, a graph that is both ptolemaic and weakly geodetic (a notion weaker than geodetic) is a block graph.
  Recall that a block graph is a graph whose every biconnected component is a clique.
  Every (connected) block graph can be constructed by the following process.
  \begin{enumerate}
  \item A single vertex is a block graph.
  \item Suppose $Y = X \cup K$, where $X$ is a block graph, $K$ is a clique, and $X\cap K$ is a single vertex. Then $Y$ is a block graph.
  \end{enumerate}
  So the diagonality of block graphs follows from Meyer-Vietoris (Hepworth and Willerton \cite{HW17} Theorem 29).
  In other words, Theorem \ref{ThmGeoPto} is a known result.
  Nevertheless, the proof using algebraic Morse theory is new and might be of interest.
  
  \bibliographystyle{alpha}
  \bibliography{ref}

\begin{thebibliography}{How81}

\bibitem[How81]{How81}
E.~Howorka.
\newblock A characterization of ptolemaic graphs.
\newblock {\em Journal of Graph Theory}, 5(3):323--331, 1981.

\bibitem[HW17]{HW17}
R.~{Hepworth} and S.~{Willerton}.
\newblock {Categorifying the magnitude of a graph}.
\newblock {\em Homology, Homotopy and Applications}, 19(2):31--60, 2017.

\bibitem[J{\"o}l05]{Jol05}
M.~J{\"o}llenbeck.
\newblock {\em Algebraic discrete Morse theory and applications to commutative
  algebra}.
\newblock PhD thesis, Philipps-Universit{\"a}t Marburg, 2005.

\bibitem[KC65]{KC65}
D.~C. Kay and G.~Chartrand.
\newblock A characterization of certain ptolemaic graphs.
\newblock {\em Canad. J. Math}, 17:342--346, 1965.

\bibitem[LC12]{LC12}
T.~Leinster and C.~A. Cobbold.
\newblock Measuring diversity: the importance of species similarity.
\newblock {\em Ecology}, 93(3):477--489, 2012.

\bibitem[Lei13]{Lei13}
T.~Leinster.
\newblock The magnitude of metric spaces.
\newblock {\em Documenta Mathematica}, 18:857--905, 2013.

\bibitem[Lei17]{Lei17}
T.~Leinster.
\newblock The magnitude of a graph.
\newblock {\em Mathematical Proceedings of the Cambridge Philosophical
  Society}, page 1–18, 2017.

\bibitem[LS17]{LS17}
T.~Leinster and M.~Shulman.
\newblock Magnitude homology of enriched categories and metric spaces.
\newblock {\em arXiv preprint arXiv:1711.00802}, 2017.

\bibitem[LV16]{LV16}
L.~Lampret and A.~Vavpeti{\v{c}}.
\newblock {(Co)homology} of lie algebras via algebraic {Morse} theory.
\newblock {\em Journal of Algebra}, 463:254--277, 2016.

\bibitem[Sk{\"o}06]{Sko06}
E.~Sk{\"o}ldberg.
\newblock Morse theory from an algebraic viewpoint.
\newblock {\em Transactions of the American Mathematical Society},
  358(1):115--129, 2006.

\end{thebibliography}
\end{document}